\documentclass[final,1p,times]{elsarticle}
\pdfoutput=1
\usepackage{amsmath}
\usepackage{amsfonts}
\usepackage{amssymb}
\usepackage{graphicx}
\usepackage{subfig}
\usepackage{amsthm}
\usepackage{float}
\usepackage[toc,page]{appendix}
\usepackage{color, colortbl}
\definecolor{LRed}{rgb}{1,.8,.8}

\newtheorem{theorem}{Theorem}[section]

\newtheorem{proposition}[theorem]{Proposition}

\newtheorem{definition}[theorem]{Definition}

\newtheorem{remark}{Remark}

\journal{Journal of Computational Physics}

\begin{document}

\begin{frontmatter}

\title{
A strategy to implement Dirichlet boundary conditions in the context of ADER finite volume schemes. One-dimensional conservation laws 
}

\author[GM]{Gino I. Montecinos\corref{cor1}\fnref{labelGM} }
\cortext[cor1]{Corresponding author}
\fntext[label2]{Email address: gmontecinos@dim.uchile.cl }
\address[GM]{        
Center for Mathematical Modeling (CMM) \\
 Universidad de Chile\\
  Santiago, Chile                
}

\begin{abstract}
ADER schemes are numerical methods, which can reach an arbitrary order of accuracy in both space and time. They are  based on a reconstruction procedure and the solution of generalized Riemann problems. However, for general boundary conditions, in particular of Dirichlet type, a lack of accuracy might occur if a suitable treatment of boundaries conditions is not properly carried out. In this work the treatment of Dirichlet boundary conditions for conservation laws in the context of ADER schemes, is concerned. The solution of generalized Riemann problems at the extremes of the computational domain, provides the correct influence of boundaries.  The reconstruction procedure, for data near to the boundaries, demands for information outside the computational domain, which is carried out in terms of ghost cells, which are provided by using the numerical solution of auxiliary problems.  These auxiliary problems are hyperbolic and they are constructed from the conservation laws and the information at boundaries, which may be   partially or totally known in terms of prescribed functions. The evolution of these  problems, unlike to the usual manner, is done in space rather than in time due to that these problems are named here, {\it reverse problems}. The methodology can be considered as a numerical counterpart of the inverse Lax-Wendroff procedure for filling ghost cells. However, the use of Taylor series expansions, as well as, Lax-Wendroff procedure, are avoided.  For the scalar case is shown that the present procedure preserve the accuracy of the scheme which is reinforced with some numerical results. Expected orders of accuracy for solving conservation laws by using  the proposed strategy at boundaries, are obtained up to fifth-order in both space and time.

\end{abstract}


\end{frontmatter}

\section{Introduction}

In this article the  treatment of Dirichle boundary conditions in the context of  high-order finite volume schemes for conservation laws is concerned. To construct  schemes of high-order  of accuracy, the reconstruction procedure is a key, for which the so-called stencils are required. These are sets of data, which allows the construction of a polynomial representation of the solution within computational cells. It is in general a well described process for cells within a computational domain, however, for cells near to boundaries is not so. In boundary treatment, it is usual,  the use of ghost cells to provided the required information outside computational domains.  Cells extrapolation techniques are usually employed to obtain ghost cells.  Lagrangian extrapolation is often used.  However, it only provides good approximations for smooth solutions. For the presence of shocks, a WENO extrapolation procedure has been proposed in \cite{Zhang:2006a}.  However, the strategy proposed by Tan and collaborators  \cite{Tan:2010a,Tan:2011a,Tan:2012a},  Huang et al.  \cite{Huang:2008a} and Xiong et al. \cite{Xiong:2010a}, seems to be the most efficient in a wide range of solution regimes. In these references, ghost cells are computed by using Taylor series expansions in space around points in the boundary. To compose these Taylor expansions the governing equation is repeatedly used  to provide space-derivatives in terms of time-derivatives similarly to the Cauchy-Kowalewskaya or Lax-Wendroff procedure, this process is known as the {\it inverse  Lax-Wendroff procedure}.
In these works, the inverse Lax-Wendroff  procedure was implemented for inflow, whereas for outflow boundaries, extrapolation techniques were concerned. This treatment has been implemented for linear problems, Vilar et al. \cite{Vilar:2014a} as well as non linear problems \cite{Filbet:2013a} to mention but a few, where  the global scheme, has resulted to be stable for any position of the boundaries.   To apply this methodology,  a manner to identify if the boundary  is an inflow or an outflow boundary,  is needed. A form to identify  and apply the type of boundary condition, is given in terms of Riemann problems. This technique is similar to that proposed by Berger et al. \cite{Berger:2003a,Berger:1990a}, where  irregular cells are incorporated near to boundaries in order to provide new cells which allow to construct  Riemann problems. In this way the inflow and outflow boundaries are automatically selected from the eigenstructure of the Riemann problem.  

In this paper, the treatment of hyperbolic conservation laws with dirichelt boundary conditions in the context of ADER finite volume schemes is concerned.  The ADER approach was first put forward by Toro et al. \cite{Toro:2001c,Toro:2002a,Titarev:2002a} for linear hyperbolic problems, see also \cite{Schwartzkopff:2002a}. Later, 
this method has also been successfully applied to solve several non-linear hyperbolic systems in Cartesian and unstructured meshes, \cite{Castro:2007a,Kaeser:2003a,Kaeser:2004a,Kaeser:2005a,Takakura:2002a,Titarev:2005a,Toro:2005a}.  ADER has also been  extended in the framework of discontinuous Galerkin finite element methods by Dumbser and collaborators, see for example \cite{Dumbser:2005a,Dumbser:2005b,Dumbser:2006a,Dumbser:2007a,Dumbser:2007c,Muller:2013b,Mueller:2014a}.  See also extensions of the ADER approach to the equations of magnetohydrodynamics by Balsara and collaborators  \cite{Balsara:2009a,Balsara:2013a} and extensions to advection-diffusion-reaction type equations \cite{Titarev:2003a,Toro:2009b,Hidalgo:2013a,Dumbser:2010a,Hidalgo:2011a,Zambra:2012a,Montecinos:2014a,Montecinos:2014d}. 
More recent results of the ADER approach include \cite{ Montecinos:2014b, Montecinos:2014c,Loubere:2014a,Dumbser:2013a,Dumbser:2014a,Dumbser:2014b,Borsche:2014a,Boscheri:2014b,Boscheri:2014c,Balsara:2016a,Montecinos:2016a}.  The ADER type methods are based on two main steps to compute numerical
solutions with arbitrary order of accuracy; i) a reconstruction
procedure and ii) the solution of a Generalized Riemann Problem (GRP).
The reconstruction procedure has to be non-linear in order to circumvent Godunov's theorem, \cite{Godunov:1959a,Toro:2009a}. This procedure at each time step requires a set of data, called stencil.  On the other hand, the solution of GRP's allows a correct wave propagation.  See \cite{Castro:2008a,Montecinos:2012a,Montecinos:2014c}  and chapters 19 and 20 of \cite{Toro:2009a} for a review of Generalized Riemann solvers.


Periodic boundary conditions, as well as, transmissive and reflective types, can easily be implemented in the context of ADER methods.
A straightforward manner to define generalised Riemann problems, as well as, to generate stencils for the
reconstruction procedure is available.  However, for general Dirichlet
boundary condition types, the reconstruction procedure requires
information outside the computational domain in order to complete
the required stencils.  Additionally, the wave propagation at boundaries must be correctly described. Therefore, in this paper  a strategy to provide the needed information in the reconstruction stage, as well as, the correct wave propagation at boundaries  through a generalised Riemann problem, are presented.  Ghost cells, which allow to create stencils, are obtained from the solution of auxiliary problems, which are constructed from the governing equations and the information at boundaries.  These auxiliary problems are hyperbolic and they may be solved by any finite volume scheme. This strategy can be considered as a numerical counter part to the inverse Lax-Wendroff procedure, but the use of Taylor series expansions and the Cauchy-Kowalewskaya procedure are avoided.
At this point we identify two types of problems, the first one,  is defined inside the computational domain and it is the problem of interest for us. This is called here, {\it interior problem}. The second type is given by the auxiliary problems which are defined outside the computational domain and allow to construct ghost cells. We remark that in opposite to the interior problem,  for auxiliary problems the evolution is carried out in space rather than in time, due to that these auxiliary problems are called {\it reverse problems}.  The methodology does not require the use of extrapolations in the usual manner and as we will show later, the methodology preserves the order of accuracy for numerical solutions of the interior problem.

The structure of the paper is as follows. In section \ref{setup}, the interior problem and the numerical scheme is discussed. In section \ref{reverse_problems} the treatment of Dirichelt boundary conditions is introduced. In section \ref{num_for_reverse} the numerical scheme for auxiliary problems is presented.  Some theoretical results concerning the scalar case are obtained in section \ref{theo_res}.  In section \ref{num_results}, numerical results are shown. Finally, conclusion are drawn in section \ref{summary}.

\section{The interior problem and the numerical scheme}\label{setup}
Let us consider hyperbolic conservation laws in the form 
\begin{eqnarray}\label{bc:eq-1}
\left.
  \begin{array}{cclc}

    \partial_t \mathbf{Q} + \partial_x \mathbf{F}(\mathbf{Q})  &=& \mathbf{0}\;, & x_L \leq x \leq x_R\;, 0\leq t \leq T\;, \\
    \mathbf{Q}(x,0)   &=& \mathbf{H}_0(x) \;, \\
  \end{array}
\right\}
\end{eqnarray}
here, $\mathbf{H}_0(x) $  is the initial condition, $\mathbf{Q}\in \mathbb{R}^m$ is the vector of unknowns,
$\mathbf{F}(\mathbf{Q})\in \mathbb{R}^m$ is the physical flux. On boundaries $x_L$ and $x_R$, inlet and outlet boundaries are observed. In the best case
prescribed functions $\mathbf{G}_L(t)\in \mathbb{R}^m$ and $\mathbf{G}_R(t)\in \mathbb{R}^m$ are provided, that means $\mathbf{Q}(x_L,t) = \mathbf{G}_L(t)$ and $\mathbf{Q}(x_R,t) = \mathbf{G}_R(t)$.  Of course, the prescribed functions influence the evolution
at a given time $t$,  according to the eigenstructure of the problem,
the number of effective boundary conditions depends on the sign of
eigenvalues.   
We remark that, in general cases, all components of $\mathbf{G}_{L}(t)$ or $\mathbf{G}_R(t)$ are not prescribed, so in the following sections we will deal the case in which component associated to outlet boundaries are not prescribed and approximate values will be provided.

Problem (\ref{bc:eq-1}) is defined inside the computational domain, due to that it is referred as {\it the interior problem}.
Let us derive the conventional finite volume formulation for (\ref{bc:eq-1}).
So, let us consider a partition of $[x_L,x_R]$ into $N_{int}$ sub intervals
$[x_{i-\frac{1}{2}},x_{i+\frac{1}{2}}]$ with uniform length $\Delta x =
x_{i+\frac{1}{2}}-x_{i-\frac{1}{2}}$,  $i=1,...,N_{int}$,  where 
$x_{-\frac{1}{2}}=x_L$ and $x_{N_{int}+\frac{1}{2}}= x_R$. Therefore, integrating the governing equation in the space-time interval $I_i^n=[x_{i-\frac{1}{2}},x_{i+\frac{1}{2}}]\times [t^n,t^{n+1}]$, we obtain the evolution formula
\begin{eqnarray}\label{bc:eq-1-1}
  \begin{array}{c}      
   \mathbf{Q}_i^{n+1} = \mathbf{Q}_i^{n} -\frac{\Delta t}{ \Delta
     x}\left[ \mathbf{F}_{i+\frac{1}{2}}-\mathbf{F}_{i-\frac{1}{2}}
   \right] \;, 
\end{array}
\end{eqnarray}
with
\begin{eqnarray}\label{bc:eq-1-2}
  \begin{array}{c}
  \mathbf{Q}_i^n = \frac{1}{\Delta x}\displaystyle \int_{x_{i-\frac{1}{2} } }^{x_{i+\frac{1}{2}}} \mathbf{Q}(x,t^n)dx\;,\; \\   
  \mathbf{F}_{i+\frac{1}{2} } =\frac{1}{\Delta t}\displaystyle
  \int_{t^n}^{t^{n+1}}\mathbf{F}( \mathbf{Q} (x_{i+\frac{1}{2} },t))dt    \;,   
 \end{array}
\end{eqnarray}
where $\mathbf{Q}_i^{n}$ is the cell average and 
$\mathbf{F}_{i+\frac{1}{2}}$ is the numerical flux.
A scheme is completely determined once the numerical flux is
defined. In this paper we are interested in the family of ADER
schemes, which can reach  an arbitrary order of accuracy in space and time.  
These are based on two steps; i) a reconstruction procedure and ii) the
solution of a Generalized Riemann Problem (GRP).

The reconstruction procedure of order $M$, provides a set of polynomials
$\mathbf{P}_i(x)$, which are constructed from sets of data, called
stencils. Thus, for example to construct the reconstruction
polynomial inside the cell $[x_{i-\frac{1}{2}},x_{i+\frac{1}{2}}]$,
the stencils are
\begin{eqnarray}
  \label{eq:stncil-0}
\begin{array}{ccc}
  \mathcal{S}_{-1,i}^{M} &=& \{
  \mathbf{Q}_{i-M}^{n},...,\mathbf{Q}_{i}^{n}\}\;,
\\
\\
  \mathcal{S}_{0,i}^{M} &=& \{
  \mathbf{Q}_{i-M}^{n},...,\mathbf{Q}_{i}^{n},...,\mathbf{Q}_{i+M}^{n}
  \}\;,
\\
\\
\mathcal{S}_{+1,i}^{M} &=& \{
\mathbf{Q}_{i}^{n},...,\mathbf{Q}_{i+M}^{n}
\}\;.
\end{array}
\end{eqnarray}
From each stencil  $\mathcal{S}_{l,i}^{M}$, with  $l=-1,0,+1$, a
polynomial of degree $M-1$ denoted by $\mathbf{P}_{i,l}(x)$, is
constructed. A polynomial reconstruction basis $\{\theta_k(\xi)\}_{k=1}^{M}$ defined in $[0,1]$, is
considered. So, each polynomial has the form
\begin{eqnarray}
  \label{poly:1}
  \mathbf{P}_{i,l}(x) = \sum_{k=1}^{M}\gamma_k^l
  \theta_k(\frac{x-x_{i-\frac{1}{2}}}{\Delta x})\;,
\end{eqnarray}
where $\gamma_k^l$ are found by solving
\begin{eqnarray}
  \label{poly:2}
\begin{array}{cc}
  \mathbf{Q}_{j}^{n} =\sum_{k=1}^{M} \gamma_k^l \int_{j}^{j+1}
  \theta(\xi)_kd\xi\;, & \mathbf{Q}_j^n\in S_{l}^{M}\;,
\end{array}
\end{eqnarray}
which is a linear problem for $\gamma_{k}^l$. For $\mathcal{S}_0^M$ we have more
elements than degrees of freedom, so an overestimated system has to be
solved, which is carried out through a constrained least-squares technique, see \cite{Dumbser:2007a}.  So, the reconstruction polynomial with support in
$[x_{i-\frac{1}{2}},x_{i+\frac{1}{2}}]$ is constructed as 
\begin{eqnarray}
  \label{poly:3}
  \mathbf{P}_i(x) =
  \omega_{-1}\mathbf{P}_{i,-1}(x)
  +\omega_{0}\mathbf{P}_{i,0}(x)
  +\omega_{+1}\mathbf{P}_{i,+1}(x)\;,
\end{eqnarray}
where $\omega_l,$ $i=-1,0,+1$ are weights, which are obtained as
follows. Firstly, we compute 
\begin{eqnarray}
  \label{poly_rev:4}
  \tilde{\omega}_{l} = \frac{\lambda_l}{(\epsilon +\sigma_l)^r}\;,
\end{eqnarray}
where $\lambda_{-1} = 1$, $\lambda_0 = 10^5$, $\lambda_{+1} = 1$, $\epsilon=10^{-14} $, $r=4$ and $\sigma_l$ is computed as
\begin{eqnarray}
  \label{poly_rev:5}
  \sigma_l =
  \sum_{r=1}^{M}\sum_{k=1}^{M} \gamma_k^l \int_{j}^{j+1}
  (\frac{d^r}{d\xi^r}\theta_k(\xi))^2d\xi\;.
\end{eqnarray}
Secondly, we normalize these values
\begin{eqnarray}
  \label{poly_rev:4}
  \omega_l= \frac{\tilde{\omega}_{l}}{\tilde{\omega}_{-1}+\tilde{\omega}_{0}+\tilde{\omega}_{+1}} \;,
\end{eqnarray}
with $l=-1,0,+1\;.$  See  \cite{Harten:1987a,Chou:2006a,Dumbser:2007a} for further details.

On the other hand,  in ADER methods the numerical flux in (\ref{bc:eq-1-2}), 
$\mathbf{F}_{i+\frac{1}{2}}$, is obtained from evaluating the integral
\begin{eqnarray}
  \label{det:1}
  \mathbf{F}_{i+\frac{1}{2}} = \int_0^1
  \mathbf{F}_h(\mathbf{q}_i(1,\tau),\mathbf{q}_{i+1}(0,\tau)) d\tau,
\end{eqnarray}
where $\mathbf{F}_{h}(\mathbf{q}_i,\mathbf{q}_{i+1}) $ denotes a
 Riemann solver,  which depends on two
arguments $\mathbf{q}_i(1,\tau)$ and $\mathbf{q}_{i+1}(0,\tau)$, which
are the high order extrapolated values of the data on the left and right side of
the interface $x_{i+\frac{1}{2}}$ at time $\tau$, respectively. In particular the Rusanov solver has been implemented in this work. Here,
the reconstruction polynomials are used to form the GRP's, as illustrated in \cite{Montecinos:2014b}, section 2.3.  In addition,  see chapters 19 and 20 of \cite{Toro:2009a} for a review of Generalized Riemann solvers and \cite{Castro:2008a,Montecinos:2012a} for a comparison of GRP solvers and references therein. 

Note that for cells near to
 boundaries, stencils should require values outside the
 computational domain.
For example, to construct the reconstruction polynomial for the first cell of the
computational domain, stencils $\mathcal{S}_{-1,1}^M$ and
$\mathcal{S}_{0,1}^M$ should contain $M$ values outside the left
boundary.  For some special boundary condition types, like periodic,
the stencils are filled with information which is available  inside the
computational domain. But for the general case we most provide these
values, in this case, these correspond to ghost cells.

The aim of this work is the treatment of boundary conditions, which is carried out in two steps. The first one, consists on the computation of the ghost cells to fill the stencils for the reconstruction procedure, the procedure is presented in detail in the following sections. The second issue regards the correct wave propagation coming from boundaries, it is carried out by using a commonly used strategy to estimate fluxes as boundaries as follows
\begin{eqnarray}
  \label{eq:boundaryflux:1}
  \begin{array}{ccc}
\displaystyle  \mathbf{F}_{-\frac{1}{2}} &=& \displaystyle \int_{0}^{1} \mathbf{F}_h(\mathbf{G}_L(t^n+\tau
  \Delta t),\mathbf{q}_1(0_+,\tau))d\tau\;, \\

\displaystyle  \mathbf{F}_{N_{int}+\frac{1}{2}}  &=& \displaystyle \int_{0}^{1} \mathbf{F}_h(\mathbf{q}_{N_{int}}(1_-,\tau),\mathbf{G}_R(t^n+\tau
  \Delta t))d\tau\;.
  \end{array} 
\end{eqnarray}
Here, $\mathbf{G}_L(t)$ and $\mathbf{G}_R(t)$ are functions defined on the interfaces, in which all their components are prescribed functions or they have to be approximated at each interval  $[t^{n}, t^{n+1} ]$.
In this way the inflow and outflow boundaries are automatically selected
by the Riemann solver.

\section{Reverse problems for  Dirichlet boundary conditions}\label{reverse_problems}

The aims here are twofold; first, provide the data at the interface to build Riemann problems, and second the computation of ghost cells, which will allow to fill stencils for the construction procedure.

\subsection{Computation of ghost cells}
Here we present a strategy to complete the stencils (\ref{eq:stncil-0})
given by 
 $\mathcal{S}_{l,i}^M$ with $l =-1,0,+1$ at the left and right extremes  $i=1,...,M$ and $i=N_{int}-M,...,N_{int}$, respectively. Let us assume that the physical flux is invertible, in the
sense that, there exists an operator $\mathbf{R}$ such that   
\begin{eqnarray}
  \label{eq:inveroper:1}
  \mathbf{R}(\mathbf{F}(\mathbf{Q}) )=\mathbf{Q}\;.
\end{eqnarray}
Therefore, we note that the governing equation (\ref{bc:eq-1}) can be
written as 
\begin{eqnarray}
  \label{eq:inveroper:2}
\partial_x \mathbf{U} + \partial_t  \mathbf{R}(\mathbf{U}) =\mathbf{0}\;,
\end{eqnarray}
with $\mathbf{U}=\mathbf{F}(\mathbf{Q})\;.$ Additionally, we assume that prescribed functions at both extremes of the domain, $\mathbf{G}_L(t)$ and $\mathbf{G}_R(t) $ are available at any time $t$.  This allows us to build the following problems
\begin{eqnarray}
  \label{eq:inveroper:3}
\left.
\begin{array}{ccc}
\partial_x \mathbf{U} + \partial_t  \mathbf{R}(\mathbf{U})
&=&\mathbf{0}\;,\; x \leq  x_L \;,\\
\mathbf{U}(x_L,t) &=& \mathbf{F}(\mathbf{G}_L(t))\;,\\ 
\end{array}
\right\}
\end{eqnarray}
and 
\begin{eqnarray}
  \label{eq:inveroper:4}
\left.
\begin{array}{ccc}
\partial_x \mathbf{U} + \partial_t  \mathbf{R}(\mathbf{U})
&=&\mathbf{0}\;,\; x \geq x_R \;,\\
\mathbf{U}(x_R,t) &=& \mathbf{F}(\mathbf{G}_R(t))\;.\\ 
\end{array}
\right\}
\end{eqnarray}
Notice that unlike to (\ref{bc:eq-1}),  the evolution of these
problems is carried out for the $x$ variable instead of $t$. Due to
that,  systems (\ref{eq:inveroper:3}) and (\ref{eq:inveroper:4}), are
called here  {\it reverse problems}.   These problems allow to know approximations of the solution of interior problems  at any position $x$ outside computational domain at any time $t$.  Additionally, the involvement of the governing equation of interior problems, incorporates a physical meaning to approximations.  Note that these are hyperbolic problems, so a wide range of solution regimes can be expected, in this sense shock waves or discontinuous behaviours of solution can be captured by these reverse problems.  Notice that any finite
volume scheme of the form
\begin{eqnarray}
  \label{eq:inveroper:5}
  \mathbf{U}_{i+1}^n = \mathbf{U}_i^n-\frac{\delta x}{\delta t}[\mathbf{R}_{n+\frac{1}{2}}-\mathbf{R}_{n-\frac{1}{2}}]\;,
\end{eqnarray}
with 
\begin{eqnarray}
\label{eq:inveroper:6}
  \begin{array}{cc}
  \mathbf{U}_i^n = \frac{1}{\delta t}\displaystyle \int_{t^{n-\frac{1}{2}}} ^{t^{n+\frac{1}{2}}} \mathbf{U}(x_i,t)dt\;,\; &   
  \mathbf{R}_{n+\frac{1}{2} } =\frac{1}{\delta  x}\displaystyle
  \int_{x_{i}}^{x_{i+1}}\mathbf{R}( \mathbf{U} (x,t^{n+\frac{1}{2}}))dx    \;,  
 \end{array}
\end{eqnarray}
can be employed to solve these reverse problems.  Of course, methods for interior problems may be different of those for reverse problems, even order of accuracy of both type of schemes can be different. 
In section \ref{num_for_reverse}, a simple second order method for reverse problems, is presented.


\begin{remark}
Notice that in this section, two assumptions have been done;  the physical flux $\mathbf{F}(\mathbf{Q})$ is assumed to be invertible; $\mathbf{G}_L(t)$ and $\mathbf{G}_R(t)$ are assumed to be prescribed. However, in general cases these assumptions are not longer valid. In the following sections these issues are concerned.

The  strategy could be considered as a numerical counterpart
of the inverse Lax-Wendroff procedure, but, the use of Taylor expansions, the Cauchy-Kowaleski procedure and extrapolations like those in \cite{Tan:2010a}, are avoided.  

\end{remark}

\subsection{A strategy to compute inverse functions $\mathbf{R}$ for the physical flux $\mathbf{F}$}\label{section-non-invertible-F}

Let us consider the physical flux $\mathbf{F}$  and a given $\mathbf{U}$, from the relationship between $\mathbf{F}$ and $\mathbf{U}$, we can find $\mathbf{Q}$ such that
\begin{eqnarray}
\begin{array}{c}
\mathbf{U} = \mathbf{F}(\mathbf{Q})\;.
\end{array}
\end{eqnarray}
This can be done by using a fixed point procedure

\begin{eqnarray}
\begin{array}{c}
\mathfrak{H}(\mathbf{Q})=\mathbf{U} - \mathbf{F}(\mathbf{Q})\;.
\end{array}
\end{eqnarray}
The Jacobian of $ \mathfrak{H}(\mathbf{Q})$ is given by
\begin{eqnarray}
\begin{array}{c}
\frac{\partial \mathfrak{H} }{\partial \mathbf{Q}} =- \mathbf{A}(\mathbf{Q})\;.
\end{array}
\end{eqnarray}
Then we can generate the iteration process
\begin{eqnarray}
\label{RP:eq-1}
\mathbf{Q}^{l+1} = \mathbf{Q}^{l}- \delta^{l}\;, 
\end{eqnarray}
where $\delta^{l}$ solves
\begin{eqnarray}
\label{increment:eq-1}
\delta^{l} = min_{\delta}||\mathbf{A}(\mathbf{Q}^{l}) \delta - \mathfrak{H}(\mathbf{Q}^l)||\;,
\end{eqnarray}
here, $l$ is an iteration index. Notice that $\mathbf{A}$ not require to be invertible. However, if $\mathbf{A} $ is invertible, then $\mathbf{\delta}$ is uniquely determined by (\ref{increment:eq-1}) and in addition, we can assume that (\ref{RP:eq-1}) converges at least locally to some $\mathbf{Q}^*$, that means,
$\mathbf{Q}^{l}\rightarrow \mathbf{
Q}^*$ then
\begin{eqnarray}
\begin{array}{c}
\mathbf{U} = \mathbf{F}(\mathbf{Q^*})\;
\end{array}
\end{eqnarray}
and thus
\begin{eqnarray}
\begin{array}{c}
\mathbf{R}(\mathbf{U}) = \mathbf{Q^*}\;.
\end{array}
\end{eqnarray}
We remark that this procedure can be carried out only if $\mathbf{R}$ cannot be provided analytically.

\subsection{Treatment of boundary conditions when only the inflow boundary is prescribed}\label{inflow:BC}
Here, the aim is to construct $\mathbf{G}_L(t) $ and $\mathbf{G}_R(t)$  in the case in which only the inflow boundary is prescribed.  This case has been well described in \cite{Tan:2010a} and the strategy described in section 2.4, can be applied to build $\mathbf{G}_L(t)$ and $\mathbf{G}_R(t)$. Alternatively, in this section we propose a simpler approach which avoid the solution of algebraic equations.  

We note that inlet and outlet boundaries are characterized by the sign of eigenvalues of the Jacobian matrix of $\mathbf{F}(\mathbf{Q})$.  By following the strategy in \cite{Tan:2010a} we express variables at boundaries  in terms of local characteristic variables.  To  do that, we are going to assume that Jacobian matrix $\mathbf{A}:= \partial \mathbf{F}(\mathbf{Q})/\partial \mathbf{Q}$ has a decomposition 
\begin{eqnarray}
\mathbf{A} = \mathbf{L}^{-1}\mathbf{\Lambda}\mathbf{L}\;,
\end{eqnarray}
with $\mathbf{\Lambda} =diag(\lambda_j)$, $j=1,...,m$. In addition let us assume that $\lambda_i<\lambda_{i+1}$. Without loss of generality,  let us focus on the left boundary so we are going to assume that there exists $m^*\leq m$ such that $\lambda_j(\mathbf{G}_L(t^n))<0$, $j=1,...,m^{*}$, whereas, $\lambda_j(\mathbf{G}_L(t^n)) > 0$, $j=m^{*}+1,...,m$. Then a local characteristic transformation can be carried out at boundaries.  So a local decomposition can be obtained as follows 
$$
\mathbf{W}_L(t) =\mathbf{L}(\mathbf{G}_{L}(t)) \mathbf{G}_{L}(t) \;,
$$
it also implies that eigenvalues can be expressed in terms of local characteristic variables. Similarly, once characteristic variables are available we can recover conserved variables as follows
\begin{eqnarray}
\label{local-cons:eq-1}
 \mathbf{G}_{L}(t)  =\mathbf{L}^{-1}(\mathbf{W}_L(t)) \mathbf{W}_L(t)  \;.
\end{eqnarray}
Since $\mathbf{G}_L(t)$ and $\mathbf{G}_R(t)$ are not always prescribed, we replace $\mathbf{G}_L(t_n)$ by $\mathbf{Q}_1^{n}$ and $\mathbf{G}_R(t_n)$ by $\mathbf{Q}_{N_{int}}^{n}$. Hence,  inflow boundaries on the left correspond to $\lambda_j(\mathbf{Q}_1^n) > 0$ and  information at boundary is required on the corresponding characteristic variables, that means $\mathbf{W}_{L,j}(t)$ needs to be defined for $j>m^*$ because this information is propagated into the computational domain.  However, boundary information associated to  $\lambda_j<0$ is not required, because the corresponding information is propagated outside the domain, and it does not influence the numerical solution of interior problems. However, we note that this information is still required for the reconstruction procedure.

In this section we provide a simple interpolation to fill information at boundaries, due to the lack of a prescribed function for  the inflow as well as for outflow boundaries.   The strategy simply is based on the use of the record of numerical solutions at interior cell near to  boundaries. At the time level $t^n$, we  carry out the following first order interpolation in time
\begin{eqnarray}
\label{interp:inverseproblems:eq-1}
\mathbf{T}(t) = \mathbf{Q}_{1}^{n-1}+\frac{t-t^{n-1}}{\Delta t} (\mathbf{Q}_{1}^{n}-\mathbf{Q}_{1}^{n-1})  \;.
\end{eqnarray}
We start the procedure  at $t^{0}$, with $\mathbf{Q}_{1}^{-1} := 3\mathbf{Q}_{1}^{0}-3\mathbf{Q}_{2}^{0}+\mathbf{Q}_{3}^{0}$, which is a third order extrapolation.  Then,  we transform $\mathbf{T}(t)$ into characteristic variables in a local sense
\begin{eqnarray}
\mathbf{\hat{T}}(t) = \mathbf{L}( \mathbf{Q}_1^{n}) \mathbf{T}(t)\;,
\end{eqnarray}
then, for component $j$ where information is not provided, we assign $\mathbf{W}_{L,j}(t)  = \mathbf{\hat{T}}_{j}(t)$. This is normal in the case of outflow, however, it can be employed for inflow boundaries when information is missed. Then the data on the left $\mathbf{G}_{L}(t)$ is completely defined by following (\ref{local-cons:eq-1}) but with $\mathbf{W}_L(t_n)$ replaced by $\mathbf{W}_1^n$, the primitive variable associated to $\mathbf{Q}_1^n$,  and thus reverse problems can be solved.  The same procedure on the right boundary condition can be carried out, and thus  $\mathbf{G}_R(t)$ can be obtained.

\section{Numerical solution of reverse problems and ghost cell computations}\label{num_for_reverse}

In this section we propose  a  second order scheme to solve the reverse problems on the right, in a small domain centred in $t\in [t^n,t^{n+1}]$, the same approach will be valid for reverse problems on the left.   This approach is based on the construction of a computational domain based on some few small cells around $t$ such that the second order of accuracy on these small intervals is comparable with the high-order accuracy of the global scheme on coarse meshes.   Let us construct a computation domain of $2\bar{M}-1$ cells. We start by considering $\bar{M}$ barycentre values 
\begin{eqnarray}
\begin{array}{c}
\tau_{n+\bar{M}}=t + n \delta t \;,\\
\end{array}
\end{eqnarray}
with $n=-(\bar{M}-1),...,\bar{M}-1$ and $\delta t = L \Delta t/(2\bar{M}-1)$, where $L$ is a constant length.  Thus we define $2\bar{M}-1$  cells $[t^{n-\frac{1}{2}},t^{n+\frac{1}{2}}]$ with $t^{n+\frac{1}{2}}=\tau_n+\frac{\delta t}{2}$. By construction the central cell contains $t$, whereas, extremes of this domain are given by $t_L = t^{-(\bar{M}-1)-\frac{1}{2}}$ and $t_R = t^{(\bar{M}-1)+\frac{1}{2}}$.  We will use a MUSCL type scheme for marching in space up to $x>x_R$ in $N$ iterations, which is carried out by considering 
\begin{eqnarray}
\begin{array}{c}
\delta x= \frac{(x-x_R) }{N}\;,
\end{array}
\end{eqnarray}
such that $ x =x_R+N \delta x$. The marching in space is achieved by using a one-step finite volume formula, as follows 
\begin{eqnarray}
\begin{array}{c}
\displaystyle \mathbf{U}_{i+1}^n = \displaystyle \mathbf{U}_{i}^n-\frac{\delta x}{\delta t} \left[\mathbf{R}_{n+\frac{1}{2}} -\mathbf{R}_{n-\frac{1}{2}}\right]\;,
\end{array}
\end{eqnarray}
with $i=1,...,N$  and $n=-(\bar{M}-1),...,\bar{M}-1$ where 
\begin{eqnarray}
\begin{array}{c}
\mathbf{U}_0^n = \frac{1}{\Delta t} \displaystyle \int_{t^{n-\frac{1}{2}}}^{t^{n+\frac{1}{2}}} \mathbf{F}( \mathbf{G}_R(t) ) dt\;.
\end{array}
\end{eqnarray}
The numerical flux is computed by using a MUSCL type scheme which is explained as follows.  We start by considering the interpolation polynomial 
\begin{eqnarray}
\label{recons:reverse-problem:eq-1}
\mathbf{P}^n(t)=\mathbf{U}_i^n  +\frac{(t-\tau_{n})}{\delta t}\mathbf{\Delta_n }  \;,
\end{eqnarray}
where the slope $\mathbf{\Delta_n}$ is obtained through the MINMOD limiter, which  in a component wise, has the form
\begin{eqnarray}
\begin{array}{c}
\mathbf{\Delta_{n}}_{j} =\left\{
\begin{array}{c}
0 \;, if \;, \mathbf{D}_j^{+,n} \mathbf{D}_j^{-,n}  \leq 0 \;,\\
\mathbf{D}_j^{-,n}\;,if\;, |\mathbf{D}_j^{-,n}| < |\mathbf{D}_j^{+,n}|\;,\\
\mathbf{D}_j^{+,n}\;,if\;, |\mathbf{D}_j^{+,n}| < |\mathbf{D}_j^{-,n}|\;,\\
\end{array}
\right. 
\end{array}
\end{eqnarray}
$j=1,...,m$, where $\mathbf{D}^{-,n}:=\mathbf{U}_i^{n}-\mathbf{U}_i^{n-1}$ and $\mathbf{D}^{+,n}:=\mathbf{U}_i^{n+1}-\mathbf{U}_i^{n}$.  Then, inside cell $[t^{n-\frac{1}{2}},t^{n+\frac{1}{2}}]$, extrapolations of the data at both interfaces are obtained as follows
\begin{eqnarray}
\begin{array}{c}
\mathbf{\bar{U}}_L :=\mathbf{P}^n(t^{n-\frac{1}{2}})=\mathbf{U}_i^{n} -\frac{1}{2}\mathbf{\Delta_i}   \;, \\
\mathbf{\bar{U}}_R :=\mathbf{P}^n(t^{n+\frac{1}{2}})=\mathbf{U}_i^{n} +\frac{1}{2}\mathbf{\Delta_i}   \;.
\end{array}
\end{eqnarray}
These extrapolated values are evolved as 
\begin{eqnarray}
\begin{array}{c}

\displaystyle \mathbf{\bar{U}}_L^n = \mathbf{\bar{U}}_L +\frac{\delta x}{2\delta t }(\mathbf{R}(\mathbf{\bar{U}}_L)-\mathbf{R}(\mathbf{\bar{U}}_{R}) )\;, 
\\
\\
\displaystyle \mathbf{\bar{U}}_R^n = \mathbf{\bar{U}}_R +\frac{\delta x}{2\delta t }(\mathbf{R}(\mathbf{\bar{U}}_L)-\mathbf{R}(\mathbf{\bar{U}}_{R}) )\;. 
\end{array}
\end{eqnarray}
Then, to compute the numerical flux values, $\displaystyle \mathbf{\bar{U}}_L^n$ and $\displaystyle \mathbf{\bar{U}}_R^n$ are interacted by using an approximate Riemann solver. Particularly
\begin{eqnarray}
\begin{array}{c}
\mathbf{R}_{n+\frac{1}{2}} =\displaystyle  \frac{1}{2}\left(\mathbf{R}(\displaystyle \mathbf{\bar{U}}_{R}^{n})+\mathbf{R}(\displaystyle \mathbf{\bar{U}}_{L}^{n})  \right)-\frac{\mu_{n+\frac{1}{2}} }{2}\left( \mathbf{\bar{U}}_{R}^{n}-\mathbf{\bar{U}}_{L}^{n} \right) \;,
\end{array}
\end{eqnarray}
that means, $\mathbf{R}_{n+\frac{1}{2}}$ is computed by using the Rusanov flux, \cite{Rusanov:1961a}.  Here $\mu_{n+\frac{1}{2}}$ is obtained as
\begin{eqnarray}
\begin{array}{c}
\mu_{n+\frac{1}{2}} = \displaystyle max\{ \mu (\mathbf{\bar{U}}_{L}^{n}) , \mu (\mathbf{\bar{U}}_{R}^{n})  \}\;,
\end{array}
\end{eqnarray}
where $\mu(\mathbf{U})=max_{j=1,...,m}|\mu_j(\mathbf{U}) |$,  with $\mu_j(\mathbf{U}) $ an eigenvalue of the Jacobian of $\mathbf{R}(\mathbf{U})$. On the extremes $t_L$ and $t_R$ we impose boundary conditions, by computing ghost cells through a third order extrapolation as follows 

\begin{eqnarray}
\label{extapolated:BC:reverse}
\begin{array}{ccc}

\mathbf{U}_i^{t_L-1}   &=& 3\mathbf{U}_i^{1}-3\mathbf{U}_i^{2}+\mathbf{U}_i^{3}\;,\\ 
\\
\mathbf{U}_i^{t_R+1} &=& 3\mathbf{U}_i^{\bar{M}}-3\mathbf{U}_i^{\bar{M}-1}+ \mathbf{U}_i^{\bar{M}-2}\;.\\ 

\end{array}
\end{eqnarray}
It allows to compute fluxes at boundaries as follows
\begin{eqnarray}
\begin{array}{ccc}

\mathbf{R}_{-(\bar{M}-1)-\frac{1}{2}} &=& \displaystyle  \frac{1}{2}\left(\mathbf{R}(\displaystyle \mathbf{U}_{i}^{t_L-1})+\mathbf{R}(\displaystyle \mathbf{U}_{i}^{-(\bar{M}-1)})  \right)-\frac{\lambda_{-(\bar{M}-1)+\frac{1}{2}} }{2}\left( \mathbf{U}_{i}^{-(\bar{M}-1)}-\mathbf{U}_{i}^{t_L-1} \right) \;, \\

\\

\mathbf{R}_{(\bar{M}-1)+\frac{1}{2}} &=& \displaystyle  \frac{1}{2}\left(\mathbf{R}(\displaystyle \mathbf{U}_{i}^{t_R+1})+\mathbf{R}(\displaystyle \mathbf{U}_{i}^{(\bar{M}-1)})  \right)-\frac{\lambda_{\bar{M}-1+\frac{1}{2}} }{2}\left( \mathbf{U}_{i}^{t_R+1} - 
\mathbf{U}_{i}^{(\bar{M}-1)} \right) \;. 

\end{array}
\end{eqnarray}
%
%
Therefore  after $N$ iterations, an approximation at $x>x_R$ is given by
\begin{eqnarray}
\label{reconstruction:eq-1:reverse}
\begin{array}{c}
\mathbf{Q}(x,t) \approx \mathbf{R}(\mathbf{U}_{N}^0) \;,
\end{array}
\end{eqnarray}
so, ghost cells for reconstruction procedure on the right extreme of $[x_L,x_R]$, are recovered as
\begin{eqnarray}
\label{eq:inveroper_revf0:7}
 \mathbf{Q}_{N_{int}+j}^n  = \frac{1}{\Delta x}
 \int_{x_{j+N_{int}-\frac{1}{2}}}^{x_{j+N_{int}+\frac{1}{2}}} \mathbf{Q}(x,t^{n} ) dx\;,
\end{eqnarray}
with $j=1,...,M$.  The same procedure applies for ghost cells on the left of the domain $[x_L,x_R]$.

\begin{remark}
To evaluate the integral (\ref{eq:inveroper_revf0:7}) we use a quadrature rule of three points.
\end{remark}

Figure \ref{fig:local_evolution} shows the sketch for $\bar{M}=N=3$.  In this case the solution is found in three steps. In the first iteration the boundary influences the first and last cell, then in the second iteration, the first and last cells  propagate their information to the neighbours inside the computational domain and thus this process continues in the third step where the process finishes and the influences of boundaries do not reach the center cell.

Concerning the stability of numerical schemes, we know that for interior problems the Courant-Friedrich-Levita (CFL) condition imposes
\begin{eqnarray}
\begin{array}{c}
\displaystyle \frac{ \Delta t}{ \Delta x}  = \displaystyle \frac{c}{ \displaystyle  \overline{\lambda}}\;,
\end{array}
\end{eqnarray}
with $c\leq 1 $ the CFL coefficient and  $ \overline{\lambda} = max_{i}(\lambda(\mathbf{Q}_i^n)) $ for $i=1,...,N_{int}$, where $\lambda(\mathbf{Q})=max_{j=1,...,m} |\lambda_j(\mathbf{Q})|$ with $\lambda_j(\mathbf{Q})$  eigenvalues of the Jacobian of $\mathbf{F}(\mathbf{Q})$ with respect to $\mathbf{Q}$.  Similarly  for reverse problems, stability in terms of the CFL condition imposes
\begin{eqnarray}
\begin{array}{c}
\displaystyle \frac{ \delta x}{ \delta t}  = \displaystyle \frac{c}{\displaystyle \overline{\mu}}\;,
\end{array}
\end{eqnarray}
with $\overline{\mu} = max_{n}(\mu(\mathbf{U}_i^n)) $, $n=1,...,2\bar{M}-1$ and $\mu(\mathbf{U}) =max_{j=1,...,m} |\mu_{j}(\mathbf{U}) |$, with $\mu_{j}(\mathbf{U})$  eigenvalues of the Jacobian of $\mathbf{R}(\mathbf{U})$.

\begin{proposition}\label{proposition:space-time:discretization}
Let $M$ be the order of accuracy of schemes for interior problems. A necessarily condition to obtain stable numerical schemes for reverse problems is that 
$$\eta :=\frac{\bar{M}}{N L}\leq c^2 \;, $$
with $L\leq M$.  
\end{proposition}

\begin{proof}

Now, we observe that  $\delta t = \frac{L \Delta t}{2\bar{M}-1} $ and $\delta x = \frac{K \Delta x}{ N}$, with $L\leq M$ and $K$ is the distance of $x$ with respect to $[x_L,x_R]$. Therefore
\begin{eqnarray}
\begin{array}{ccc}
\frac{c}{\overline{\mu}} &=& \frac{\delta x}{ \delta t} = \frac{ \Delta x }{\Delta t}\frac{  K(2\bar{M}-1)}{ L N }\;,\\
\\
\frac{c}{\overline{\mu}} &=& \frac{\overline{\lambda}}{ c} \frac{ K}{L}\frac{ 2\bar{M}-1}{N}\;,\\
\end{array}
\end{eqnarray}
after some manipulations
\begin{eqnarray}
\begin{array}{c}
\frac{c^2}{\overline{\mu} \overline{\lambda} } = ( \frac{ K}{L})\frac{ 2\bar{M}-1}{N}\;.\\
\end{array}
\end{eqnarray}
On the other hand, let us assume that all eigenvalues $\mathbf{A}(\mathbf{Q}) = \partial \mathbf{F}(\mathbf{Q})/\partial \mathbf{Q}$ are distinct from zero. Then
\begin{eqnarray}
\partial_t \mathbf{Q}+\mathbf{A}(\mathbf{Q}) \partial_x \mathbf{Q} = \mathbf{0}\;
\end{eqnarray}
and
\begin{eqnarray}
\partial_x \mathbf{Q}+\mathbf{A}(\mathbf{Q})^{-1} \partial_t \mathbf{Q} = \mathbf{0}\;,
\end{eqnarray}
thus we have 
\begin{eqnarray}
\begin{array}{c}
\mathbf{A}(\mathbf{Q})^{-1}  = \partial \mathbf{R}(\mathbf{U})/\partial \mathbf{U}\;. 
\end{array}
\end{eqnarray}
Now, let us define $ \underline{\lambda} =min_j|\lambda_j(\mathbf{Q})|>0$.  Then
\begin{eqnarray}
\begin{array}{c}
1\leq \frac{\overline{\lambda}}{\underline{\lambda}}= \overline{\lambda}\overline{\mu}\;.
\end{array}
\end{eqnarray}
Therefore
%
\begin{eqnarray}
\begin{array}{c}
c^2\geq \frac{c^2}{\overline{\mu} \overline{\lambda} } = ( \frac{ K}{L})\frac{ 2\bar{M}-1}{N}\geq \frac{ \bar{M}}{N L}\;.\\
\end{array}
\end{eqnarray}
Thus the result holds.
\end{proof}

\begin{remark}
Reverse problems are hyperbolic conservation laws. Hence, any high-order numerical scheme may be applied. However, high-order schemes reduce dramatically the efficiency of global solvers.  As we will see later, in proposition \ref{accuracy} and  numerical results in section \ref{num_results}, for numerical implementations second order of accuracy for reverse problems should be enough to get the accuracy for interior problems.  
\end{remark}

\begin{figure}
\begin{center}
\includegraphics[scale=0.8]{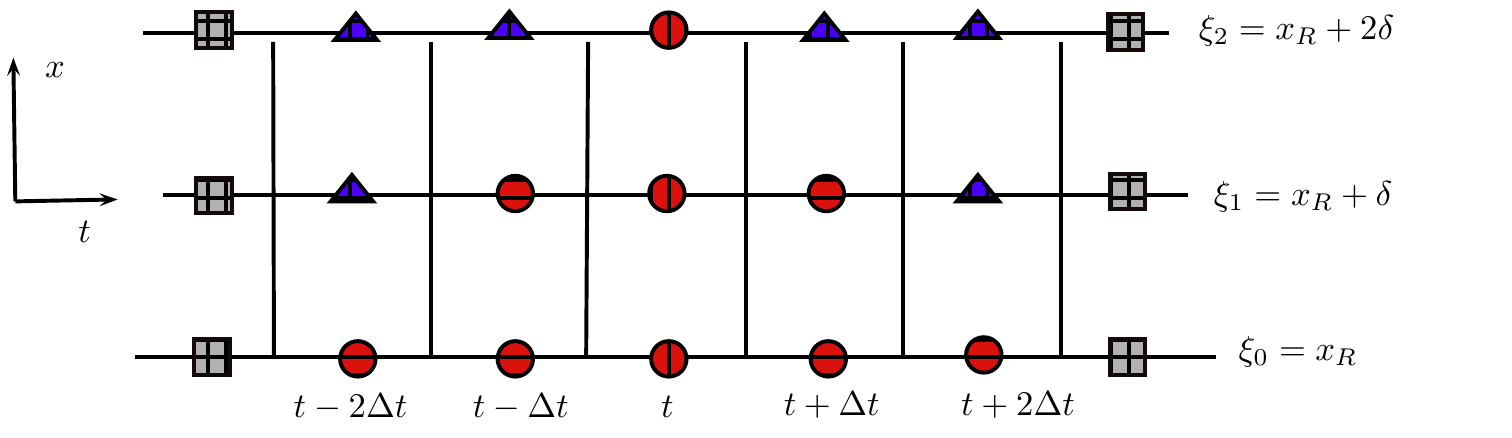}
\end{center}
\caption{Sketch of local evolution around $t$ for three cells $\bar{M}=N=3$. Squares reperesent the ghost cells, circles and triangles are interior points evolved with a second order scheme. Solution is reached in $N=3$ iterations. Triangles are interior points which at each evolution, are influencied for the information coming from boundaries.  Circles are the interior cells which only depend on the interior point.  }\label{fig:local_evolution}
\end{figure}

\section{Some theoretical results}\label{theo_res}
In this section we are going to present some theoretical results for the scalar case
\begin{eqnarray}
\label{eq:1}
\left.
\begin{array}{c}

\partial_t q(x,t) + \partial_x f(q(x,t)) = 0 \;, x\in [x_L,x_R]\;,\\

q(x,0)    = h(x)\;, \\
q(x_L,t)  = g_L(t)\;, \\
q(x_R,t)  = g_R(t)\;. \\
\end{array}
\right\}
\end{eqnarray}
The reverse problems (\ref{eq:inveroper:3}) and (\ref{eq:inveroper:4}) associated to (\ref{eq:1}) are formed by the governing equation
\begin{eqnarray}
\label{eq:2}
\partial_x u(x,t)+\partial_t R(u(x,t)) = 0 \;, 
\end{eqnarray}
where $R(u) = q $.  Here, we are going to prove that the present methodology preserves the theoretical order of accuracy for interior problems and for the case in which the physical flux contains an inverse function,  the solution of the reverse  problem can be tracked from the interior problem through the characteristic curves.  In what follows, we will say {\it reverse boundary condition} to indicate that Dirichelt boundary conditions are included through the solution of a GRP at the extremes of the computational domain and the solution of reverse problems is used to provide the ghost cells in order to fill the stencils and thus to carry out the reconstruction procedure.
  
To start,  let us assume that  problem associated to (\ref{eq:1}) has the exact solution $\bar{q}(x,t)$ and problem associated to (\ref{eq:2}) has the exact solution $\bar{u}(x,t)$.  Additionally, we are going to assume that the numerical solutions for the interior problem  are obtained with a numerical scheme of accuracy  $r$ in space and time.  Let us assume that the scheme for solving  (\ref{eq:1}) can be written as 
\begin{eqnarray}
\label{eq:3}
q_i^{n+1} = \sum_{j=-k_L}^{k_R} b_j(q_{i+j-k_L}^n,..,q_{i+j}^n,...,q_{i+j+k_R}^n)q_{i+j}^n\;,
\end{eqnarray}
where $b_j$ is a function of the data. These functions can be non-linear. 
Then if we adopt the definition of accuracy as in \cite{Ben-Artzi:2003a}, (definition 2.21).  The scheme applied to the exact solution, at least for cells inside the computational domain satisfies
\begin{eqnarray}
\label{eq:4}
\begin{array}{c}
\bar{q}(x_i,t^{n+1}) - \sum_{j=-k_L}^{k_R} b_j(\bar{q}(x_{i+j-k_L},t^n),...,\bar{q}(x_{i+j+k_R},t^n))\bar{q}(x_{i+j},t^n) \\= O(\Delta t^{r+1})\;. 
\\
\end{array}
\end{eqnarray}
As the time step $\Delta t $ and the mesh size $\Delta x$ are related through a CFL condition, the accuracy is simultaneously the same in space and time.  So, the aim here is to probe that  (\ref{eq:4}) is also valid for values near to boundaries. This is proved in the following proposition.

\begin{proposition}\label{accuracy}
Let $r$ be the order of accuracy of the numerical scheme for the interior problem. If reverse problems are solved with a numerical scheme of order of accuracy $p$. Then 
the numerical scheme for the interior problem  maintains the expected theoretical order of accuracy even at boundaries. If 
\begin{eqnarray}
\begin{array}{c}
 N \geq  \Delta x^{\frac{p-r}{p+1}} r \;,
\end{array}
\end{eqnarray}
where $N$ is defined in proposition \ref{proposition:space-time:discretization}.

\end{proposition}

\begin{proof}
Without loss of generality, we are going to prove that (\ref{eq:4}) is attained  for the first left computational cell. Then,  in order to take into account the way in which the data outside the computational domain is obtained, we write   (\ref{eq:3}) as
\begin{eqnarray}
\label{eq:5}
\begin{array}{c}
q_1^{n+1} = \sum_{j=-k_L}^{k_R} b_j(D(1+j-k_L,q),..,D(0,q),...,D(1+j+k_R,q) )\times \\ D(1+j,q) \; , 
\\ 
\\
\end{array}
\end{eqnarray}
with  
\begin{eqnarray}
\label{eq:5-1}
D(i,q)=\left\{
\begin{array}{cc}
R(P_i^n(t^n)) \;, i<1\;,\\
q_i^n \;,i\geq 1\;.
\end{array}
\right.
\end{eqnarray}
Here,  $R(P_i^n(t^n))$  denotes  the solution of the reverse problem at $t^n$, 
obtained as in  (\ref{reconstruction:eq-1:reverse}) and $P_i^n(t)$ denotes the reconstruction  polynomial (\ref{recons:reverse-problem:eq-1}). On the other hand, as the reverse problem is solved with a numerical scheme of accuracy $p$, in which $\delta t$ and $\delta x$ are related through a CFL condition, we have that the numerical and exact solutions, $u_i^n$  and $\bar{u}(x_i,t^n)$ respectively, are related as follows
\begin{eqnarray}
\label{eq:6}
u_i^n = \bar{u}(x_i,t^n) + O(\delta x^{p+1}) + O(\delta t^{p+1})=\bar{u}(x_i,t^n) + O(\delta t^{p+1})\;.
\end{eqnarray}
Additionally, we also  assume that the reconstruction polynomial is of order $p$ in the following sense
\begin{eqnarray}
 P_i^n(t^n) = u_i^n + O(\delta t^{p+1}) \;. 
\end{eqnarray}
Then from (\ref{eq:6}) one has
\begin{eqnarray}
 P_i^n(t^n) = u_i^n  + O(\delta t^{p+1}) =  \bar{u}(x_i,t^n) + O(\delta t^{p+1})\;. 
\end{eqnarray}
Thus, we can write  
\begin{eqnarray}
\label{eq:7}
\begin{array}{c}
 R(P_i^n(t^n)) = \bar{q}(x_i,t^n)  + O(\delta t^{p+1}) R'(P_i^n(t^n))+ O(\delta t^{2(p+1)})  \\= \bar{q}(x_i,t^n)  + O(\delta t^{p+1})  \;. 
\end{array}
\end{eqnarray}
Therefore, from (\ref{eq:5-1}) and (\ref{eq:7}), the numerical scheme applied to the exact solution $\bar{q}(x,t)$ provides
\begin{eqnarray*}
\begin{array}{c}
\bar{q}(x_1,t^{n+1}) = \displaystyle  \sum_{j=-k_L}^{-1} b_j(\bar{q}(x_{1+j-k_L},t^n)  + O(\delta t^{p+1})  ,...,\bar{q}(x_{1+j+k_R},t^n)  + O(\delta t^{p+1})  )\times \\ (\bar{q}(x_{1+j},t^n)  + O(\delta t^{p+1})) \;  \\ 
\displaystyle +\sum_{j=0}^{k_R} b_j(D(1+j-k_L,\bar{q}),..,D(0,\bar{q}),...,\bar{q}(x_{1+j+k_R},t^n) )\bar{q}(x_{1+j},t^n)   \\
= 
 \displaystyle  \sum_{j=-k_L}^{k_R} b_j(\bar{q}(x_{1+j-k_L},t^n) ,...,\bar{q}(x_{1+j+k_R},t^n)  ) \bar{q}(x_{1+j},t^n)+  O(\delta t^{p+1})  \Phi\;,
\end{array}
\end{eqnarray*}
with 
\begin{eqnarray}
\Phi =  \big( \sum_{j=-k_L}^{k_R} b_j +\sum_{k=1+j<1}b_{j,k}O(\delta t^{p+1})   \big) \;,
\end{eqnarray}
where $b_{j,k} = \partial b_j/\partial q_k$.  For simplicity we have dropped the arguments of $b_j$ and its derivatives $b_{j,k}.$ As we are free of choosing $\delta t$, we take, $\delta t$ such that,  $O(\delta t^{p+1}) \Phi  = O(\Delta t^{r+1})$. In virtue of proposition \ref{proposition:space-time:discretization}, we note that
\begin{eqnarray}
\begin{array}{c}
O(\delta x) =O(\frac{ K}{ N}\Delta x)= O( \Delta x ^{ \frac{r +1}{p +1} } )\;,
\end{array}
\end{eqnarray}
hence, $ O(N)= O( \Delta x ^{ \frac{p -r}{p +1} } K ) $ with $K$ defined  given in Proposition \ref{proposition:space-time:discretization}and thus  it is related with the required stencils, which for order $r$ is exactly $K=r$. So, by using the meaning of $O(\cdot)$, the result holds.


\end{proof}

\begin{remark}

Notice that the previous result does not consider a particular scheme. We note that for values of $\Delta x < 1$, it is required $p\leq r$ to obtain a feasible disctretization of reverse problems, that means, $N >1 $. 

\end{remark}

Now, let us see how are related the solution of the interior problem and the solution of reverse problems. For that we need $f$ to have an inverse function which is ensured in the scalar case by the following.
\begin{proposition}
Let $f(q)$ be a differentiable function in $\mathbb{R}$ with $f(q)'\neq 0$, then there exists $R$ such that $R(f(q))=q.$ 
\end{proposition}
\begin{proof}
The result follows from the inverse function Theorem, see \cite{krantz:2002a} and references therein. 
\end{proof}
Now we are going to prove that if the flux  $f(q)$ contains an inverse function $f^{-1}$, the solution of reverse problems can be tracked from the interior problem through characteristic curves. To start let us consider the following.
\begin{definition}
Let $w(x,t)$ be the exact solution of 
\begin{eqnarray}\label{BCHO:eq-3}
\left.
\begin{array}{cccc}
\partial_t q + \partial_x f(q(x)) &=& 0 \;, & x \in \mathbb{R}\;,
\\
q(x,0) &=& h(x) \;.
\end{array}
\right\}
\end{eqnarray}
We say that boundary conditions in (\ref{eq:1}) are compatible if $g_L(t)=w(x_L,t)$ and $g_R(t)=w(x_R,t)$.
\end{definition}
Now we are going to prove that the solution of the reverse problems under compatibility conditions, are contained in the characteristic curves of the interior problems. It is carried out in the following.

\begin{proposition}\label{prop2}
Let $w(x,t)$ be the exact solution of (\ref{BCHO:eq-3}). Let $\bar{q}(x,t)$ be the exact solution of problem (\ref{eq:1}) endowed with compatible boundary conditions, where   $g_L(t)$ and $g_R(t)$ are prescribed functions at boundaries and let $\bar{u}(x,t)$ be the exact solution of the reverse problems.
If $R \equiv f^{-1}$.  Then $\bar{q}(x,t)=w(x,t)$ in $[x_L,x_R]$ and $f^{-1}(\bar{u}(x,t)) = w(x,t) $ for all $x\leq x_L$ and $x\geq x_R$.  
\end{proposition}

\begin{proof}
Without loss of generality we are going to consider the left reverse problem, defined for $x \leq x_L$ and let us denote by  $\bar{u}(x,t)$ its solution.  We note that on the curve $t_1(x)$ which is defined by
\begin{eqnarray}
\left.
\begin{array}{ccc}
\displaystyle \frac{d t_1(x)}{dx } &=& (f^{-1}(\bar{u}(x,t_1(x)) ) )' \;,\\

t_1(x_L) &=& \delta\;,
\end{array}
\right\}
\end{eqnarray}
the following holds 
\begin{eqnarray}
\begin{array}{c}
\displaystyle \frac{d \bar{u}(x,t_1(x))}{dx} = 0\;.
\end{array}
\end{eqnarray}
On the other hand, from the exact solution $w(x,t)$ we define the curve $t_2(x)$ given by
\begin{eqnarray}
\left.
\begin{array}{ccc}
\displaystyle \frac{d t_2(x)}{dx } &=& \displaystyle (f'(w(x,t_2(x)))^{-1} \;,\\

\displaystyle t_2(x_L) &=& \delta\;,
\end{array}
\right\}
\end{eqnarray}
along which the following is satisfied
\begin{eqnarray}
\begin{array}{c}
\displaystyle\frac{d f(w(x,t_2(x)))}{dx} = 0\;.
\end{array}
\end{eqnarray}
On the other hand, by following these curves  and from the compatibility condition  
\begin{eqnarray}
\begin{array}{c}
\bar{u}(x,t_1(x))= \bar{u}(x_L,t_1(x_L))=f(g_L( \delta ))=f(g_L(t_2(x_L)))=f(w(x_L,t_2(x_L)))=f(w(x,t_2(x))) \;.
\end{array}
\end{eqnarray}
From this equality we have
\begin{eqnarray}
\begin{array}{c}
\displaystyle \frac{d t_1(x)}{dx} = (f^{-1}(\bar{u}(x,t_1(x))))'=(f'(w(x,t_2(x))))^{-1}= \displaystyle \frac{d t_2(x)}{dx}\;,
\end{array}
\end{eqnarray}
this means that $t_1(x)-t_2(x)=t_1(x_L)-t_2(x_L)=0,$ then both curves coincide, which allows to define $t(x):=t_1(x)=t_2(x)$.  Thus for $x<x_L$ we have  $f^{-1}(\bar{u}(x,t))= w(x,t).$

On the other hand, inside the computational domain $[x_L,x_R]$ let us define the following curves $x_1(t)$ and $x_2(t)$, defined by the ODE's
\begin{eqnarray}
\label{eq:charq:1}
\left.
\begin{array}{ccc}
\displaystyle \frac{d x_1(x)}{dt } &=& f'(\bar{q}(x_1(t),t) \;,\\
x_1(0) &=& y\;,
\end{array}
\right\}
\end{eqnarray}
and
\begin{eqnarray}
\label{eq:charw:1}
\left.
\begin{array}{ccc}
\displaystyle \frac{d x_2(x)}{dt } &=& f'(w(x_2(t),t) \;,\\
x_2(0) &=& y\;.
\end{array}
\right\}
\end{eqnarray}
These are the respective characteristic curves, along which $\bar{q}(x,t)$ and $w(x,t)$ remain constant. This yields
\begin{eqnarray}
\label{eq:wq-1}
\begin{array}{ccc}
\bar{q}(x_1(t),t)=\bar{q}(x_1(0),0)=h(y)=w(x_2(0),0)=w(x_2(t),t)\;.
\end{array}
\end{eqnarray}
From (\ref{eq:charq:1}), (\ref{eq:charw:1}) and (\ref{eq:wq-1}), we obtain $x(t):=x_1(t)=x_2(t)$. Therefore in $[x_L,x_R]$,   $\bar{q}(x,t)=w(x,t)$ and thus the result holds.
%
\end{proof}
In the following section we will solve interior problems in which the present methodology for Dirichlet boundary condition is implemented.

\section{Numerical results}\label{num_results}
 

In all tests, the ADER-DET method is implemented to solve interior problems and the Rusanov solver \cite{Rusanov:1961a}, is used to solve the Riemann problems. See \cite{Dumbser:2008a} for further details.


\subsection{Linear advection}
Let us consider the interior problem given by the linear advection equation 
\begin{eqnarray}
  \label{lae}
\left.
\begin{array}{ccc}
  \partial_t q +\partial_x(\lambda q) &= &0 \;, x\in [0,1],t\in (0,T]\;,\\
  q(x,0) &=& sin(2\pi x)\;.\\
\end{array}
\right\}
\end{eqnarray}
In this case we note that the left boundary is an inflow boundary where a prescribed function $g_L(t)$ is assumed to be available, whereas, the right boundary is an outflow where no boundary is required. However, in order to apply our methodology we  construct the function $g_R(t)$ as suggested in section \ref{inflow:BC},  such that  $q(1,t)=g_R(t)$.  If $g_L(t) =-sin(2\pi\lambda t) $, the exact solution is given by $q(x,t) = sin(2\pi(x-\lambda t))$. 
Therefore, reverse problems are easily obtained, so the left reverse problem is given by
\begin{eqnarray}
  \label{lae_rev:1}
\left.
  \begin{array}{ccc}
    \partial_x u +\partial_t(\lambda^{-1} u) &=& 0\;\;, x< x_L\;,  \\
    u(t,x_L) &=& -\lambda \sin(2\pi\lambda t)\;, \\
  \end{array}
\right\}
\end{eqnarray}
whereas, right reverse problem is given by 
\begin{eqnarray}
  \label{rae_rev:1}
\left.
  \begin{array}{ccc}
    \partial_x u +\partial_t(\lambda^{-1} u) &=& 0\;\;, x> x_R\;,  \\
    u(t,x_R) &=&  u_{N_{int}}^{n-1} +\frac{t-t^{n}}{\Delta t} (u_{N_{int}}^{n}-u_{N_{int}}^{n-1}) \;. \\
  \end{array}
\right\}
\end{eqnarray}
Extrapolations like (\ref{extapolated:BC:reverse}) are applied to deal with boundary conditions for reverse problems. Table \ref{LAtest:tab-1}, shows the results of a systematically
convergence rates assessment  for the interior problem with $\lambda =1$,  CFL equals to 0.9 and
$t_{out}  = 4$. For reverse problems we have used $N=20$, $\bar{M}=10$ and $L = 0.7$, which provides $\eta = \bar{M}/NL = 0.71< c^2 =(0.9)^2 =0.81 $, thus from the Proposition \ref{proposition:space-time:discretization}, a stable scheme is expected.   We observe that the accuracy is achieved up to fifth order of accuracy.  In order to compare the performance of the present procedure, we compare the corresponding CPU time against the CPU time of the inverse Lax-Wendroff procedure. Table \ref{LAtest:tab-2:CPU}, shows the CPU time comparison between the present method and inverse Lax-Wendroff procedure. We observe for second, third and fourth orders of accuracy the performance of both boundary treatments is similar. However, for fifth order the CPU of both procedures have  same magnitude.
\begin{table}
\begin{center}
Theoretical order : 2 \\
\begin{tabular}{cccccccc} 
\\
\hline
\hline 
Mesh  & $L_\infty$ - err & $L_\infty$- ord  & $L_1$ - err & $L_1$ - ord & $L_2$ - err & $L_2$ - ord & CPU  \\  
\hline

     8  &  0.00  &$  6.98e-0 2$&  0.00  &$  3.88e-0 2$&  0.00  &$  4.30e-0 2$  &0.0240\\
    16  &  1.21  &$  3.02e-0 2$&  1.96  &$ 10.00e-0 3$&  1.72  &$  1.30e-0 2$  &0.0520\\
    32  &  1.25  &$  1.27e-0 2$&  1.92  &$  2.64e-0 3$&  1.66  &$  4.13e-0 3$  &0.1320\\
    64  &  1.25  &$  5.33e-0 3$&  2.07  &$  6.28e-0 4$&  1.68  &$  1.29e-0 3$  &0.4080\\
   128  &  1.29  &$  2.19e-0 3$&  2.14  &$  1.42e-0 4$&  1.75  &$  3.82e-0 4$  &1.0880\\

 \hline
 \\
 \end{tabular}  
\\
Theoretical order : 3 \\
\begin{tabular}{cccccccc}  
\\
\hline 
Mesh  & $L_\infty$ - err & $L_\infty$- ord  & $L_1$ - err & $L_1$ - ord & $L_2$ - err & $L_2$ - ord & CPU  \\  
\hline

     8  &  0.00  &$  2.01e-0 2$&  0.00  &$  9.88e-0 3$&  0.00  &$  1.14e-0 2$  &0.0160\\
    16  &  3.27  &$  2.09e-0 3$&  3.20  &$  1.08e-0 3$&  3.20  &$  1.24e-0 3$  &0.0400\\
    32  &  2.53  &$  3.61e-0 4$&  2.82  &$  1.52e-0 4$&  2.76  &$  1.83e-0 4$  &0.1160\\
    64  &  2.62  &$  5.85e-0 5$&  2.96  &$  1.96e-0 5$&  2.93  &$  2.41e-0 5$  &0.3160\\
   128  &  3.53  &$  5.06e-0 6$&  3.11  &$  2.26e-0 6$&  3.14  &$  2.74e-0 6$  &1.1200\\

 \hline
 \\
\end{tabular} 
\\
Theoretical order : 4 \\
\begin{tabular}{cccccccc}  	
\\
\hline 
Mesh  & $L_\infty$ - err & $L_\infty$- ord  & $L_1$ - err & $L_1$ - ord & $L_2$ - err & $L_2$ - ord & CPU  \\  
\hline

     8  &  0.00  &$  2.23e-0 2$&  0.00  &$  8.64e-0 3$&  0.00  &$  1.14e-0 2$  &0.0200\\
    16  &  5.41  &$  5.22e-0 4$&  5.21  &$  2.34e-0 4$&  5.29  &$  2.91e-0 4$  &0.0520\\
    32  &  3.74  &$  3.90e-0 5$&  4.69  &$  9.03e-0 6$&  4.53  &$  1.26e-0 5$  &0.1280\\
    64  &  4.11  &$  2.26e-0 6$&  3.81  &$  6.42e-0 7$&  3.92  &$  8.34e-0 7$  &0.3920\\
   128  &  3.93  &$  1.49e-0 7$&  3.74  &$  4.80e-0 8$&  3.77  &$  6.10e-0 8$  &1.3400\\

 \hline
\\
\end{tabular} 
\\
Theoretical order : 5 \\
\begin{tabular}{cccccccc}  
\\
\hline 
Mesh  & $L_\infty$ - err & $L_\infty$- ord  & $L_1$ - err & $L_1$ - ord & $L_2$ - err & $L_2$ - ord & CPU  \\  
\hline

     8  &  0.00  &$  9.49e-0 3$&  0.00  &$  4.87e-0 3$&  0.00  &$  5.38e-0 3$  &0.0920\\
    16  &  5.95  &$  1.53e-0 4$&  5.83  &$  8.55e-0 5$&  5.83  &$  9.45e-0 5$  &0.1880\\
    32  &  4.80  &$  5.50e-0 6$&  5.16  &$  2.39e-0 6$&  5.15  &$  2.67e-0 6$  &0.4200\\
    64  &  4.84  &$  1.93e-0 7$&  5.33  &$  5.92e-0 8$&  5.22  &$  7.16e-0 8$  &1.0480\\
   128  &  5.12  &$  5.54e-0 9$&  5.38  &$  1.42e-0 9$&  5.36  &$  1.74e-0 9$  &3.1440\\

 \hline
\end{tabular} 
\end{center}
\caption{Convergence rates for the linear advection at output time
  $t_{out} = 4$ with  $C_{cfl}= 0.9,$ $\lambda = 1$, $N=20 $, $\bar{M}=10$ and $L = 0.7$. Left reverse boundary conditions are applied. Inflow boundary on the left boundary and outflow boundary on the right boundary.}\label{LAtest:tab-1}
\end{table}
\begin{table}
\begin{center}

Theoretical order : 2 \\
\begin{tabular}{ccc} 
\\
 \hline
N & CPU reverse problem & CPU inverse Lax-Wendroff \\  
 \hline
  \hline
8   & 0.0240 & 0.020  \\  
16  & 0.0520 & 0.064  \\    
32  & 0.1320 & 0.240  \\
64  & 0.4080 & 0.244  \\
128  & 1.0880 & 0.756  \\            
 \hline
\end{tabular} 
\end{center}

\begin{center}
Theoretical order : 3 \\
\begin{tabular}{ccc} 
\\
 \hline
N & CPU reverse problem & CPU inverse Lax-Wendroff \\  
 \hline
  \hline
8   & 0.016 & 0.052  \\  
16  & 0.0400 & 0.056  \\    
32  & 0.1160 & 0.060  \\
64  & 0.3160 & 0.244  \\
128  & 1.1200 & 0.908  \\            
 \hline
\end{tabular} 
\end{center}

\begin{center}
Theoretical order : 4 \\
\begin{tabular}{ccc} 
\\
 \hline
N & CPU reverse problem & CPU inverse Lax-Wendroff \\  
 \hline
  \hline
8   & 0.0200 & 0.080  \\  
16  & 0.0520 & 0.020  \\    
32  & 0.1280 & 0.088  \\
64  & 0.3920 & 0.320  \\
128  & 1.3400 & 1.352  \\            
 \hline
\end{tabular} 
\end{center}

\begin{center}

Theoretical order : 5 \\
\begin{tabular}{ccc} 
\\
 \hline
N & CPU reverse problem & CPU inverse Lax-Wendroff \\  
 \hline
  \hline
8   & 0.0920 & 0.012  \\  
16  & 0.1880 & 0.040  \\    
32  & 0.4200 & 0.172  \\
64  & 1.0480 & 0.540  \\
128  & 3.1440 & 2.084  \\
 \hline            
\end{tabular} 
\end{center}

\caption{Linear advection. CPU time comparisons by orders of accuracy. Reverse problems (second columns) and inverse Lax-Wendroff procedure (third columns).}\label{LAtest:tab-2:CPU}
\end{table}

\subsection{Hyperbolic system with a non-invertible Jacobian}

In this section we deal with the issue of a non-invertible Jacobian matrix. To construct this test, let us consider the scalar case
\begin{eqnarray}
\label{zero_eig:eq-1}
\begin{array}{c}
\partial_t q(x,t) +\partial_x(a(x) q(x,t)) = 0\;,x \in [0,1]\;,
\end{array}
\end{eqnarray}
as we note the physical flux depends on $x$, in such a case Riemann problem can non be solved as conventional and the accuracy can be penalized but also the stability of numerical schemes. See   \cite{Zhang:2003a,Zhang:2005a} for further details on space-dependent fluxes. To overcome any difficulty arising from space-dependent fluxes, we can transform (\ref{zero_eig:eq-1}) into the following system
\begin{eqnarray}
\label{zero_eig:eq-2}
\begin{array}{c}
\partial_t \mathbf{Q}+\partial_x\mathbf{F}( \mathbf{Q})=\mathbf{0}\;,
\end{array}
\end{eqnarray}
where 
\begin{eqnarray}
\begin{array}{c}
\mathbf{Q}=
\left[
\begin{array}{c}
q\\
a
\end{array}
\right]\;,

\mathbf{F}(\mathbf{Q})=
\left[
\begin{array}{c}
aq\\
0
\end{array}
\right]
\;.
\end{array}
\end{eqnarray}
The Jacobian of $\mathbf{F}(\mathbf{Q})$ is given by
\begin{eqnarray}
\begin{array}{c}
\mathbf{A} =
\left[
\begin{array}{cc}
a & q \\
0 & 0
\end{array}
\right]
\end{array}
\end{eqnarray}
and eigenvalues of $\mathbf{A}(\mathbf{Q})$ are $\lambda_1 = 0$ and $\lambda_2 = a$.  

We can construct an exact solution for (\ref{zero_eig:eq-1}) and thus to (\ref{zero_eig:eq-2}). This can be done by inserting in  (\ref{zero_eig:eq-1}), the functions 
\begin{eqnarray}
\label{zero_eig:eq-3}
\begin{array}{c}
\hat{q}(x,t) = sin(\frac{x}{t+1})\;, \\
\hat{a}(x)   = exp(-50(x-0.5)^2)\;,
\end{array}
\end{eqnarray}
which produces a new equation 
\begin{eqnarray}
\label{zero_eig:eq-4}
\begin{array}{c}
\partial_t q(x,t) +\partial_x(a(x) q(x,t)) = s(x,t)\;,
\end{array}
\end{eqnarray}
with $s(x,t) = \hat{q}_t +\hat{a}'\hat{q} + \hat{a}\hat{q}_x$.  Clearly (\ref{zero_eig:eq-3}) solves (\ref{zero_eig:eq-4}) and of course, also solves 
\begin{eqnarray}
\label{zero_eig:eq-2-2}
\begin{array}{c}
\partial_t \mathbf{Q}+\partial_x\mathbf{F}( \mathbf{Q})=\mathbf{S}(x,t)\;,
\end{array}
\end{eqnarray}
with $\mathbf{S}=[s(x,t),0]^T$.
We apply reverse problems to system (\ref{zero_eig:eq-2}) thus we need a function $\mathbf{R}(\mathbf{U})$ such that 
$$
\mathbf{R}(\mathbf{U})=\mathbf{Q}\;.
$$
Notice that in this case, vectors $\mathbf{U}$ have the form $\mathbf{U} = [u,0]^T$. Therefore an inverse function can be given by
\begin{eqnarray}
\label{zero_eig:eq-5}
\begin{array}{c}
\mathbf{R}(\mathbf{U}) = 
\left[
\begin{array}{c}
u/k \\
k
\end{array}
\right]\;,
\end{array}
\end{eqnarray}
with $k$ a constant value.  It can be verified that $\mathbf{F}(\mathbf{R}(\mathbf{U}))=\mathbf{U}$.
Therefore the reverse problem on the left has the form
\begin{eqnarray}
\begin{array}{c}
\partial_x \mathbf{U} + \partial_x \mathbf{R}(\mathbf{ \mathbf{U}}) = \mathbf{0}\;, x<x_L \;,\\
\mathbf{U}(0,t) = \mathbf{F}(\mathbf{\hat{Q}}(0,t))\;,
\end{array}
\end{eqnarray}
with 
\begin{eqnarray}
\begin{array}{c}
\hat{Q}(x,t)=
\left[
\begin{array}{c}
\hat{q}(x,t) \\
\hat{a}(x) \\
\end{array}
\right]
\;,
\end{array}
\end{eqnarray}
whereas, the reverse problem on the right is given by %
\begin{eqnarray}
\begin{array}{c}
\partial_x \mathbf{U} + \partial_x \mathbf{R}(\mathbf{ \mathbf{U}}) = \mathbf{0}\;, x>x_R \;,\\
\mathbf{U}(1,t) = \mathbf{F}(\mathbf{\hat{Q}}(1,t))\;.
\end{array}
\end{eqnarray}

\begin{figure}
\begin{center}
\includegraphics[scale=0.5]{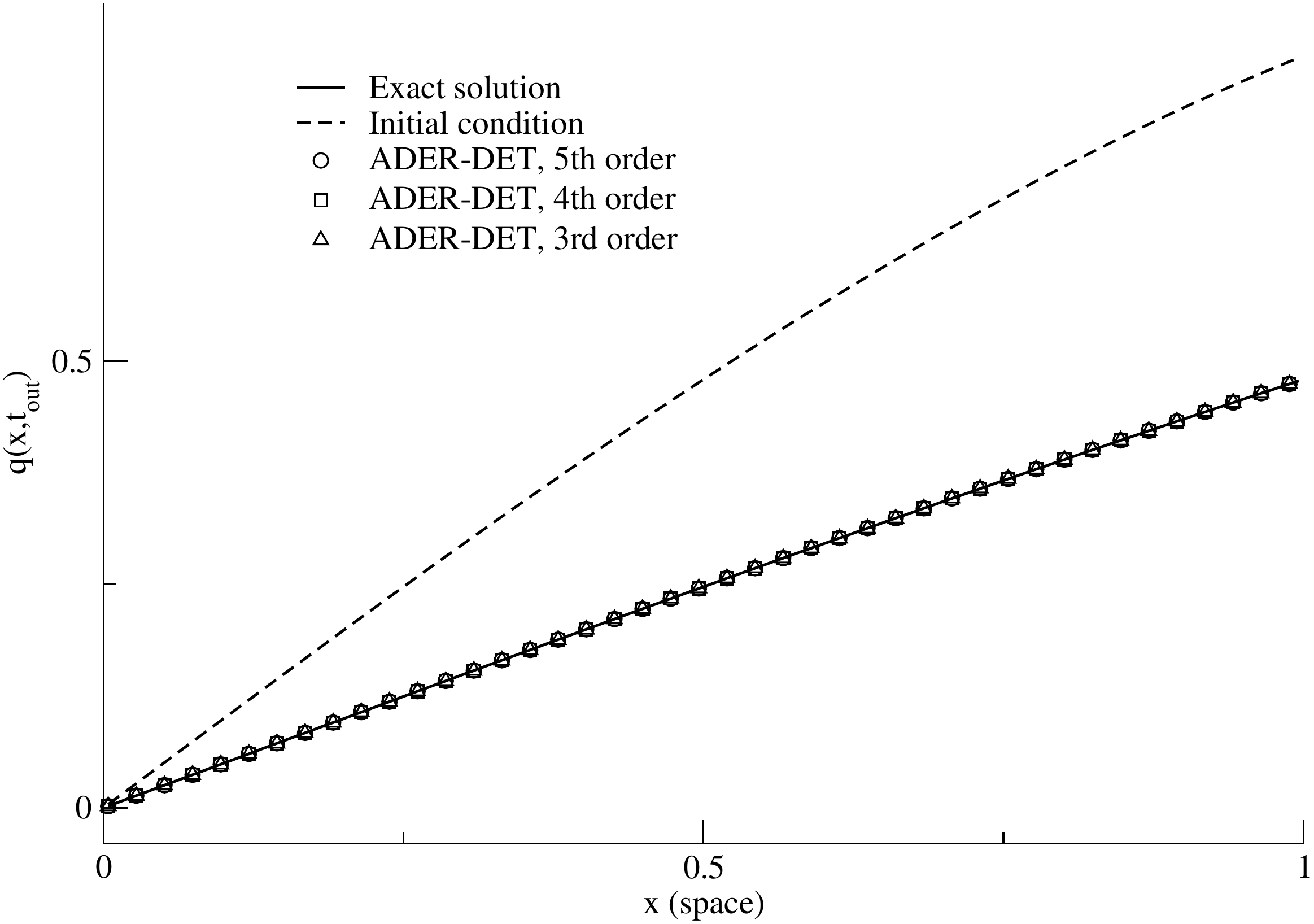}
\end{center}
\caption{Non-invertible Jacobian matrix. Initial condition (Dashed line),  exact solution (full line), third order solution (triangles),  fourht order solution (squares) and fifth order solution (circles). Parameters: $N = \bar{M} = 3$, $L = 2$, $k=1$, 128 cells,  $t_{out} = 1$ and $C_{cfl} = 0.9$}\label{fig:zero_eigenvalue}
\end{figure}
\begin{table}
\begin{center}
Theoretical order : 2 \\
\begin{tabular}{cccccccc} 
\\
\hline
\hline 
Mesh  & $L_\infty$ - err & $L_\infty$- ord  & $L_1$ - err & $L_1$ - ord & $L_2$ - err & $L_2$ - ord & CPU  \\  
\hline

     8  &  0.00  &$  2.10e-0 1$&  0.00  &$  1.12e-0 1$&  0.00  &$  1.31e-0 1$  &0.0080\\
    16  &  0.88  &$  1.13e-0 1$&  1.64  &$  3.58e-0 2$&  1.41  &$  4.94e-0 2$  &0.0200\\
    32  &  1.01  &$  5.64e-0 2$&  1.72  &$  1.09e-0 2$&  1.55  &$  1.68e-0 2$  &0.0520\\
    64  &  2.49  &$  1.00e-0 2$&  2.65  &$  1.73e-0 3$&  2.62  &$  2.73e-0 3$  &0.1560\\
   128  &  2.15  &$  2.25e-0 3$&  2.62  &$  2.80e-0 4$&  2.50  &$  4.83e-0 4$  &0.5000\\

 \hline
 \\
 \end{tabular}  
\\
Theoretical order : 3 \\
\begin{tabular}{cccccccc}  
\\
\hline 
Mesh  & $L_\infty$ - err & $L_\infty$- ord  & $L_1$ - err & $L_1$ - ord & $L_2$ - err & $L_2$ - ord & CPU  \\  
\hline

     8  &  0.00  &$  5.80e-0 1$&  0.00  &$  1.46e-0 1$&  0.00  &$  2.25e-0 1$  &0.0040\\
    16  &  3.37  &$  5.60e-0 2$&  2.83  &$  2.06e-0 2$&  3.00  &$  2.82e-0 2$  &0.0200\\
    32  &  0.90  &$  3.01e-0 2$&  1.94  &$  5.36e-0 3$&  1.71  &$  8.62e-0 3$  &0.0680\\
    64  &  2.93  &$  3.96e-0 3$&  2.84  &$  7.50e-0 4$&  2.84  &$  1.21e-0 3$  &0.1800\\
   128  &  2.84  &$  5.51e-0 4$&  2.96  &$  9.65e-0 5$&  2.94  &$  1.57e-0 4$  &0.6520\\
 
 \hline
 \\
\end{tabular} 
\\
Theoretical order : 4 \\
\begin{tabular}{cccccccc}  	
\\
\hline 
Mesh  & $L_\infty$ - err & $L_\infty$- ord  & $L_1$ - err & $L_1$ - ord & $L_2$ - err & $L_2$ - ord & CPU  \\  
\hline

     8  &  0.00  &$  1.08e-0 0$&  0.00  &$  2.58e-0 1$&  0.00  &$  4.10e-0 1$  &0.0080\\
    16  &  0.03  &$  1.06e-0 0$&  0.95  &$  1.34e-0 1$&  0.48  &$  2.94e-0 1$  &0.0320\\
    32  &  6.13  &$  1.51e-0 2$&  5.26  &$  3.48e-0 3$&  5.80  &$  5.26e-0 3$  &0.0840\\
    64  &  4.60  &$  6.22e-0 4$&  4.61  &$  1.43e-0 4$&  4.63  &$  2.13e-0 4$  &0.2880\\
   128  &  4.29  &$  3.18e-0 5$&  4.72  &$  5.40e-0 6$&  4.71  &$  8.11e-0 6$  &1.0600\\
 
 \hline
 \\
\end{tabular} 
\\
Theoretical order : 5 \\
\begin{tabular}{cccccccc}  	
\\
\hline 
Mesh  & $L_\infty$ - err & $L_\infty$- ord  & $L_1$ - err & $L_1$ - ord & $L_2$ - err & $L_2$ - ord & CPU  \\  
\hline

     8  &  0.00  &$  1.14e-0 0$&  0.00  &$  2.10e-0 1$&  0.00  &$  4.16e-0 1$  &0.0360\\
    16  &  4.63  &$  4.61e-0 2$&  4.31  &$  1.06e-0 2$&  4.74  &$  1.55e-0 2$  &0.0960\\
    32  &  4.54  &$  1.98e-0 3$&  4.35  &$  5.17e-0 4$&  4.35  &$  7.60e-0 4$  &0.2320\\
    64  &  4.61  &$  8.11e-0 5$&  4.91  &$  1.72e-0 5$&  4.83  &$  2.68e-0 5$  &0.6840\\
   128  &  5.05  &$  2.46e-0 6$&  4.94  &$  5.58e-0 7$&  4.94  &$  8.70e-0 7$  &2.2080\\
 
 \hline
\end{tabular} 
\end{center}
\caption{Convergence rates for space-dependent flux at output time
  $t_{out} = 1$ with  $C_{cfl}= 0.9$ and $k=1$ for the inverse function $\mathbf{R}(\mathbf{U})$. Reverse boundary conditions are applied with $N=3$, $\bar{M} = 3$ and $L = 2$.}\label{non-inv:tab-1}
\end{table}
Table \ref{non-inv:tab-1}, shows the result of the convergence rate assessment by using reverse problems for boundary conditions. We have used the combination $N=\bar{M}=3$ and $L = 2$, which gives $\eta = 0.5$ and so in virtue of proposition \ref{proposition:space-time:discretization}, it produces a globally stable scheme.  Notice that we use the term {\it globally stable scheme} to say that schemes for interior problems as well as that for reverse problems are stable. 

We observe that expected orders of accuracy are achieved up to fifth order of accuracy.  As we note there are infinite numbers of inverse functions for flux $\mathbf{F}$, for each $k$ $\mathbf{R}(\mathbf{U})$ is a suitable inverse function. However, we can check that by following the procedure in section \ref{section-non-invertible-F}, the starting guess $\mathbf{Q}^{0} = [q_1 , q_2]^T$ yields  the converged state $\mathbf{Q}^{*} = [q_1,1]^T$, hence the choice of  $k = 1$  in (\ref{zero_eig:eq-5}) agrees with the expected result in section \ref{section-non-invertible-F}, due to that this is set for the simulations.

\subsection{Linear system model}
In this test, the interior problem is given by the linear system of
conservation laws
\begin{eqnarray}
  \label{laesys:1}
\begin{array}{cc}
  \partial_t \mathbf{Q} +\partial_x\mathbf{F}( \mathbf{Q}) = 0 \;,& x\in [0,1]\;,\\
  \mathbf{Q}(x,0) = \mathbf{H}_0(x)\;,\\
  \mathbf{Q}(0,t) = \mathbf{G}_L(x)\;,\\
  \mathbf{Q}(1,t) = \mathbf{G}_R(x)\;,\\
\end{array}
\end{eqnarray}
with 
\begin{eqnarray}
  \label{laesys:2}
\begin{array}{cc}

\mathbf{F}(\mathbf{Q})=
\left[
\begin{array}{c}
q_1-q_2\\
2q_2
\end{array}
\right]\;,
&
\mathbf{H}_0(x)=
\left[
\begin{array}{c}
sin(2\pi x)\\
cos(2\pi x)
\end{array}
\right]\;.
\end{array}
\end{eqnarray}
The exact solution is
\begin{eqnarray}
  \label{exactsys}
  \begin{array}{c}
    \mathbf{Q}^e(x,t) =
\left[
  \begin{array}{c}
    sin(2\pi(x-t))+cos(2\pi (x-t))-2cos(2\pi(x-2t))\\
4cos(2\pi(x-2t))
  \end{array}
\right],
  \end{array}
\end{eqnarray}
so, similarly to the scalar linear case $\mathbf{G}_L(t)$ and
$\mathbf{G}_R(t)$ are given by the exact solution evaluated at $x=0$ and
$x=1$, respectively. 

In this case the reverse problems are directly obtained. Without loss of generality, the left reverse problem is given by
\begin{eqnarray}
  \label{slae_rev:1}
\left.
  \begin{array}{ccc}
    \partial_x \mathbf{U} + \partial_t(\mathbf{R}( \mathbf{U})) &=& 0\;\;,
x< x_L\;,  \\
    \mathbf{U}(t,x_L) &=& \mathbf{F}(\mathbf{G}_L(t)) \;, \\
  \end{array}
\right\}
\end{eqnarray}
where extrapolated  boundary conditions are applied on the extremes of the computational domain. Here 
\begin{eqnarray}
  \label{laesys_rev:2}
\begin{array}{cc}
\mathbf{U}=
\left[
\begin{array}{c}
u_1\\
u_2
\end{array}
\right]\;,
&
\mathbf{R}(\mathbf{U})=
\left[
\begin{array}{c}
u_1+\frac{u_2}{2}\\
\frac{u_2}{2}
\end{array}
\right]\;.
\end{array}
\end{eqnarray} 
Table \ref{LStest:tab-1} shows the error and the
result of the empirical convergence rate assessment for the interior problem with the combination $N = 50$, $\bar{M} = 20$, $L = 2$, which produces  $\eta = 0.2$, so a globally stable method should be obtained. Additionally, we use $t_{out} = 1$ and
$CFL = 0.9$. We observe that expected order of accuracy is achieved up to fifth order of accuracy. 

\begin{table}
\begin{center}
Theoretical order : 2 \\
\begin{tabular}{cccccccc} 
\\
\hline
\hline 
Mesh  & $L_\infty$ - err & $L_\infty$- ord  & $L_1$ - err & $L_1$ - ord & $L_2$ - err & $L_2$ - ord & CPU  \\  
\hline

    8  &  0.00  &$  5.73e-0 1$&  0.00  &$  2.82e-0 1$&  0.00  &$  3.36e-0 1$  &0.5360\\
    16  &  1.24  &$  2.43e-0 1$&  1.73  &$  8.48e-0 2$&  1.55  &$  1.14e-0 1$  &0.5240\\
    32  &  2.65  &$  3.87e-0 2$&  2.05  &$  2.04e-0 2$&  2.31  &$  2.31e-0 2$  &1.0120\\
    64  &  1.80  &$  1.12e-0 2$&  2.01  &$  5.06e-0 3$&  2.05  &$  5.58e-0 3$  &2.1000\\
   128  &  2.16  &$  2.49e-0 3$&  2.04  &$  1.23e-0 3$&  2.03  &$  1.36e-0 3$  &4.3240\\

 \hline
 \\
\end{tabular}  
\\
Theoretical order : 3 \\
\begin{tabular}{cccccccc} 
\\
\hline 
Mesh  & $L_\infty$ - err & $L_\infty$- ord  & $L_1$ - err & $L_1$ - ord & $L_2$ - err & $L_2$ - ord & CPU   \\  
\hline

     8  &  0.00  &$  3.46e-0 1$&  0.00  &$  1.64e-0 1$&  0.00  &$  2.00e-0 1$  &0.2600\\
    16  &  2.45  &$  6.33e-0 2$&  2.42  &$  3.07e-0 2$&  2.46  &$  3.63e-0 2$  &0.3400\\
    32  &  2.93  &$  8.28e-0 3$&  2.88  &$  4.16e-0 3$&  2.91  &$  4.81e-0 3$  &0.6920\\
    64  &  3.01  &$  1.03e-0 3$&  3.00  &$  5.21e-0 4$&  3.00  &$  6.00e-0 4$  &1.4120\\
   128  &  3.01  &$  1.28e-0 4$&  3.01  &$  6.48e-0 5$&  3.01  &$  7.45e-0 5$  &3.0720\\

 \hline
\\
\end{tabular}  
\\
Theoretical order : 4 \\
\begin{tabular}{cccccccc} 
\\
\hline 
Mesh  & $L_\infty$ - err & $L_\infty$- ord  & $L_1$ - err & $L_1$ - ord & $L_2$ - err & $L_2$ - ord & CPU  \\  
\hline

     8  &  0.00  &$  5.71e-0 1$&  0.00  &$  2.97e-0 1$&  0.00  &$  3.41e-0 1$  &0.2560\\
    16  &  4.46  &$  2.60e-0 2$&  4.40  &$  1.40e-0 2$&  4.38  &$  1.64e-0 2$  &0.5400\\
    32  &  4.72  &$  9.84e-0 4$&  4.77  &$  5.14e-0 4$&  4.81  &$  5.84e-0 4$  &1.0440\\
    64  &  3.83  &$  6.92e-0 5$&  4.18  &$  2.84e-0 5$&  4.20  &$  3.17e-0 5$  &2.1600\\
   128  &  3.82  &$  4.90e-0 6$&  3.93  &$  1.87e-0 6$&  3.91  &$  2.10e-0 6$  &4.7560\\
 
 \hline

\end{tabular}
\\
Theoretical order : 5 \\
\begin{tabular}{cccccccc} 
\\
\hline 
Mesh  & $L_\infty$ - err & $L_\infty$- ord  & $L_1$ - err & $L_1$ - ord & $L_2$ - err & $L_2$ - ord & CPU  \\  
\hline

     8  &  0.00  &$  2.33e-0 1$&  0.00  &$  1.24e-0 1$&  0.00  &$  1.44e-0 1$  &6.8160\\
    16  &  5.17  &$  6.47e-0 3$&  5.07  &$  3.69e-0 3$&  5.08  &$  4.28e-0 3$  &13.3720\\
    32  &  5.07  &$  1.92e-0 4$&  4.98  &$  1.17e-0 4$&  5.02  &$  1.32e-0 4$  &26.9999\\
    64  &  4.97  &$  6.13e-0 6$&  4.96  &$  3.75e-0 6$&  4.97  &$  4.22e-0 6$  &54.7440\\
   128  &  4.97  &$  1.96e-0 7$&  4.97  &$  1.20e-0 7$&  4.97  &$  1.34e-0 7$  &110.7640\\
    
 \hline

\end{tabular} 
\end{center}
\caption{Convergence rates for the linear system at output time
  $t_{out} = 1$ with  $C_{cfl}= 0.9\;.$ Reverse boundary conditions are applied with $N=50$, $\bar{M}=10$ and $L= 2$.}\label{LStest:tab-1}
\end{table}

\subsection{The Euler equations}
Now let us consider an interior problem in which the problem is non-linear, so we consider the Euler equations, characterized by
\begin{eqnarray}
\label{eq:1-euler}
\begin{array}{ccc}
\mathbf{Q} =
\left[
\begin{array}{c}
\rho \\
\rho u \\
E
\end{array}
\right]\;,
&
\mathbf{F}(\mathbf{Q}) =
\left[
\begin{array}{c}
\rho u \\
\rho u^2 + p \\
u(E+p)
\end{array}
\right]\;.
\end{array}
\end{eqnarray}
Here, the pressure $p$ is related with the conserved variables through
the equation for an ideal gas with $\gamma = 1.4$, which yields
\begin{eqnarray}
\label{eq:2-euler}
p = (\gamma -1) (E -\frac{\rho u^2}{2}) \;.
\end{eqnarray}
The initial condition for this system in terms of
non-conservative variables $[\rho, u,p]$,  is given by
\begin{eqnarray}
\begin{array}{c}
\rho(x,0) = 1+0.2\sin(2\pi x)\;,\\
u(x,0)= 1,\\
p(x,0)= 2,
\end{array}
\end{eqnarray}
additionally the system is endowed with Dirichlet boundary conditions in terms of
non-conservative variables, given by
\begin{eqnarray}
\begin{array}{cc}
\mathbf{W}_L(t) =
\left[
\begin{array}{c}
1-0.2\sin(2\pi t)\\
1\\
2
\end{array}
\right]  \;,
&
\mathbf{W}_R(t) =
\left[
\begin{array}{c}
1+0.2\sin(2\pi(1- t))\\
1\\
2
\end{array}
\right]  \;.
\end{array}
\end{eqnarray}
So, the exact solution is given by 
\begin{eqnarray}
\begin{array}{ccc}
\rho(x,t) &=& 1+0.2\sin(2\pi (x-t))\;,\\
u(x,t)    &=& 1\;,\\
p(x,t)    &=& 2\;.
\end{array}
\end{eqnarray}
The reverse problems in this case are not straightforward obtained. Thus in order to provide the form of the reverse problems, we define
\begin{eqnarray}
\phi(\mathbf{U}) = \frac{2u_2+\sqrt{4u_2^2+8(\gamma^2-1) 
         \biggl(\frac{u_2^2}{2}-u_1u_3\biggr)  }}{2(\gamma+1)} \;.
\end{eqnarray}
Therefore, the  flux for the reverse problems is given by
\begin{eqnarray}
\begin{array}{c}
\mathbf{R}(\mathbf{U})=
\left[
\begin{array}{c}
     \displaystyle \frac{u_1^2}{u_2-\phi(\mathbf{U})}\\
     u_1\\
    \displaystyle  \frac{u_2}{2}+\frac{\phi(\mathbf{U})(3-\gamma )}{ 2\gamma-1)}
\end{array}
\right]\;.
\end{array}
\end{eqnarray}
Table \ref{Eulertest:tab-1} shows the convergence rate  assessment for
the Euler equations. Here we have used $t_{out} = 1$ and
$C_{CFL}=0.9$. For simulations, reverse problems have been applied with $N=\bar{M}=3$ and $L=1.5$ and so $\eta = 0.66$, from proposition \ref{proposition:space-time:discretization}, the scheme should be globally stable.  We observe that the expected theoretical orders of accuracy for interior problems are achieved.
Figure \ref{eulerdensity} shows the comparison between the exact  and numerical solution of fourth order 
of accuracy.  The numerical solution has been computed with 64 cells.     
\begin{table}
\begin{center}
Theoretical order : 2\\
\begin{tabular}{cccccccc} 
\\
 \hline
\hline 
Mesh  & $L_\infty$ - err & $L_\infty$- ord  & $L_1$ - err & $L_1$ - ord & $L_2$ - err & $L_2$ - ord  & CPU \\  
\hline

    32  &  0.00  &$  7.35e-0 3$&  0.00  &$  3.81e-0 3$&  0.00  &$  4.41e-0 3$  &0.3680\\
    64  &  2.02  &$  1.81e-0 3$&  2.16  &$  8.51e-0 4$&  2.15  &$  9.93e-0 4$  &0.3880\\
   128  &  2.08  &$  4.29e-0 4$&  2.09  &$  2.00e-0 4$&  2.08  &$  2.35e-0 4$  &1.0160\\
   256  &  2.04  &$  1.04e-0 4$&  2.04  &$  4.87e-0 5$&  2.04  &$  5.72e-0 5$  &3.1960\\
   512  &  2.00  &$  2.60e-0 5$&  2.02  &$  1.20e-0 5$&  2.02  &$  1.41e-0 5$  &11.4280\\

 \hline
\\
\end{tabular}
\\ 
Theoretical order : 3 \\
\begin{tabular}{cccccccc} 
\\
\hline 
Mesh  & $L_\infty$ - err & $L_\infty$- ord  & $L_1$ - err & $L_1$ - ord & $L_2$ - err & $L_2$ - ord & CPU  \\  
\hline

    32  &  0.00  &$  1.92e-0 3$&  0.00  &$  8.99e-0 4$&  0.00  &$  1.05e-0 3$  &0.1440\\
    64  &  3.01  &$  2.38e-0 4$&  2.95  &$  1.16e-0 4$&  2.97  &$  1.35e-0 4$  &0.4400\\
   128  &  3.01  &$  2.96e-0 5$&  2.98  &$  1.48e-0 5$&  2.99  &$  1.69e-0 5$  &1.1320\\
   256  &  3.00  &$  3.70e-0 6$&  2.99  &$  1.86e-0 6$&  3.00  &$  2.12e-0 6$  &3.8480\\
   512  &  3.00  &$  4.62e-0 7$&  3.00  &$  2.33e-0 7$&  3.00  &$  2.65e-0 7$  &14.1720\\

 \hline
 \\
\end{tabular}
\\ 
Theoretical order : 4 \\
\begin{tabular}{cccccccc}
\\ 
\hline 
Mesh  & $L_\infty$ - err & $L_\infty$- ord  & $L_1$ - err & $L_1$ - ord & $L_2$ - err & $L_2$ - ord & CPU  \\  
\hline

    32  &  0.00  &$  5.00e-0 4$&  0.00  &$  1.70e-0 4$&  0.00  &$  2.14e-0 4$  &5.9920\\
    64  &  5.79  &$  9.01e-0 6$&  5.22  &$  4.56e-0 6$&  5.37  &$  5.17e-0 6$  &14.5119\\
   128  &  3.86  &$  6.21e-0 7$&  3.95  &$  2.95e-0 7$&  3.93  &$  3.40e-0 7$  &33.5880\\
   256  &  3.90  &$  4.16e-0 8$&  3.91  &$  1.96e-0 8$&  3.91  &$  2.27e-0 8$  &89.7599\\
   512  &  3.95  &$  2.70e-0 9$&  3.94  &$  1.27e-0 9$&  3.94  &$  1.48e-0 9$  &259.4400\\    
 
 \hline
 \\
\end{tabular} 
\\ 
Theoretical order : 5 \\
\begin{tabular}{cccccccc}
\\ 
\hline 
Mesh  & $L_\infty$ - err & $L_\infty$- ord  & $L_1$ - err & $L_1$ - ord & $L_2$ - err & $L_2$ - ord & CPU  \\  
\hline

    32  &  0.00  &$  3.56e-0 5$&  0.00  &$  1.92e-0 5$&  0.00  &$  2.16e-0 5$  &47.8040\\
    64  &  4.99  &$  1.12e-0 6$&  4.94  &$  6.26e-0 7$&  4.94  &$  7.02e-0 7$  &103.8680\\
   128  &  4.99  &$  3.50e-0 8$&  4.95  &$  2.02e-0 8$&  4.96  &$  2.26e-0 8$  &223.6119\\
   256  &  5.00  &$  1.10e-0 9$&  4.97  &$  6.43e-010$&  4.98  &$  7.20e-010$  &502.5120\\
   512  &  4.99  &$  3.45e-011$&  4.99  &$  2.03e-011$&  4.99  &$  2.27e-011$  &1154.5119\\
 
 \hline

\end{tabular} 

\end{center}
\caption{Convergence rates for the Euler equations at output time
  $t_{out} = 1.$ with  $C_{cfl}= 0.9\;.$  Reverse boundary conditions are applied with $N=\bar{M}=3$ and $L=1.5$.}\label{Eulertest:tab-1}
\end{table}
Table \ref{Eulertest:cpu:tab-1} shows the CPU-time comparison, between the reverse problems and the inverse Lax-Wendroff procedure. We note that the CPU time has the same order of accuracy for both procedures.  The Lax-Wendroff procedure, used in this test is described in \ref{section:inverse-LW}. 
\begin{table}
\begin{center}
Theoretical order : 2\\
\begin{tabular}{ccc} 
\\
 \hline
\hline 
Mesh  & CPU reverse problems & CPU inverse Lax-Wendroff \\ 
\hline

    32  &  0.3680  & 0.2720 \\
    64  &  0.3880  & 0.17199 \\
   128  &  1.0160  & 0.6599 \\
   256  &  3.1960  & 2.6239 \\
   512  &  11.4280 & 9.8200 \\

 \hline
\\
\end{tabular}
\\
Theoretical order : 3\\
\begin{tabular}{ccc} 
\\
 \hline
\hline 
Mesh  & CPU reverse problems & CPU inverse Lax-Wendroff \\ 
\hline

    32  &  0.1440  & 0.2440 \\
    64  &  0.4400  & 0.3079 \\
   128  &  1.1320  & 1.0999 \\
   256  &  3.8480  & 3.9199 \\
   512  &  14.1720 & 14.5480 \\

 \hline
\\
\end{tabular}
\\
Theoretical order : 4\\
\begin{tabular}{ccc} 
\\
 \hline
\hline 
Mesh  & CPU reverse problems & CPU inverse Lax-Wendroff \\ 
\hline

    32  &  5.9920    & 5.4359 \\
    64  &  14.5119   & 13.8479 \\
   128  &  33.5880   & 29.1920 \\
   256  &  89.7599   & 78.0840 \\
   512  &  259.4400  & 255.1160  \\    

 \hline
 \\
\end{tabular}
\\
Theoretical order : 5\\
\begin{tabular}{ccc} 
\\
 \hline
\hline 
Mesh  & CPU reverse problems & CPU inverse Lax-Wendroff \\ 
\hline

    32  &  47.8040   & 48.1220 \\
    64  &  103.8680  & 97.5440 \\
   128  &  223.6119  & 215.7440 \\
   256  &  502.5120  & 476.1519 \\
   512  &  1154.5119 & 1086.2359 \\

 \hline
\\
\end{tabular}

\end{center}
\caption{CPU for the Euler equations at output time
  $t_{out} = 1.$ with  $C_{cfl}= 0.9\;.$  Reverse boundary conditions with $N=\bar{M}=3$ and $L=1.5$ (second colums) and inverse Lax-Wendroff procedure for boundaries (third columns). }\label{Eulertest:cpu:tab-1}
\end{table}
\begin{figure}
\includegraphics[scale=0.5]{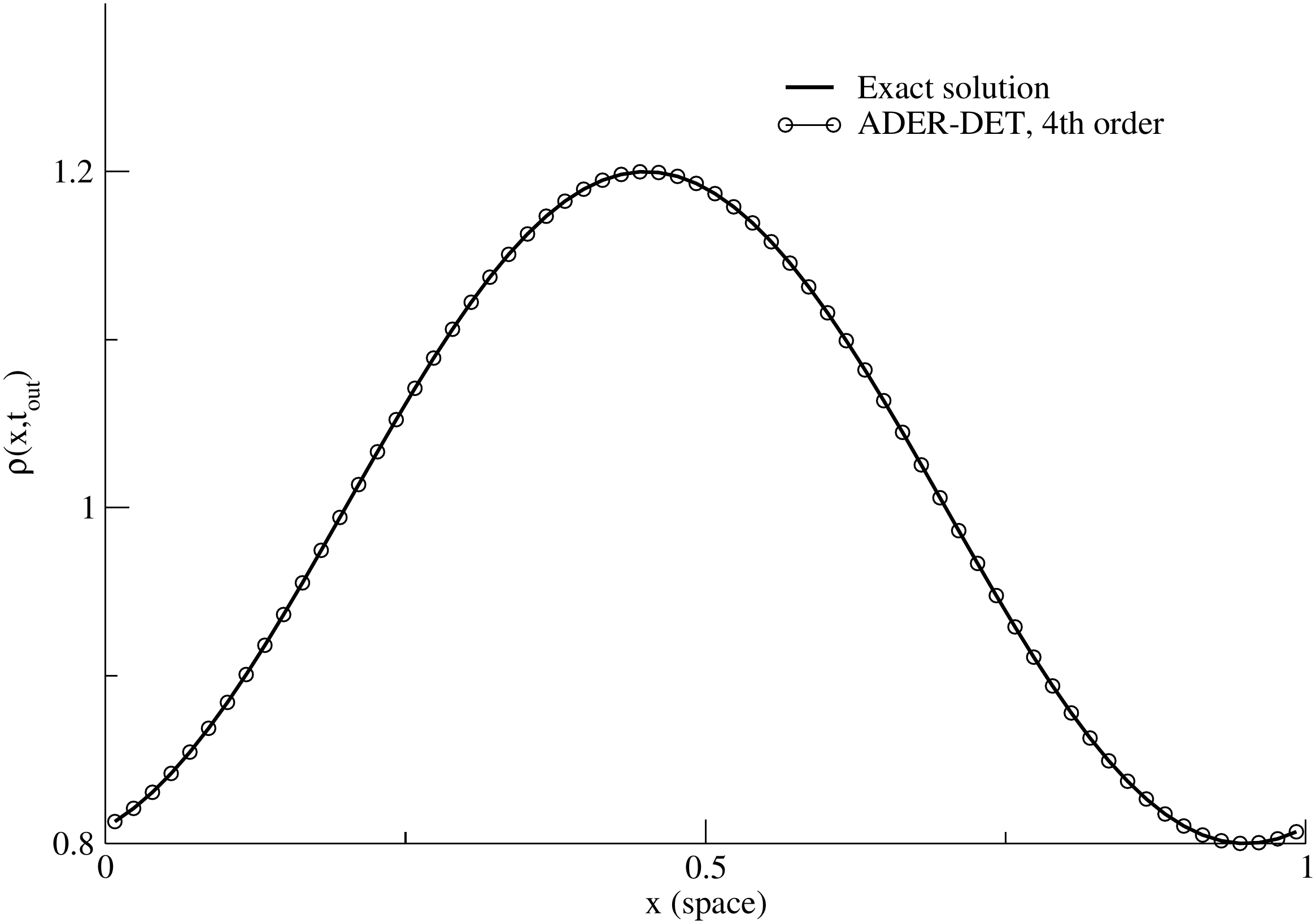}
\caption{Density for the Euler equation. Comparison between exact and numerical solutions with 64 cells. Reverse boundary conditions have been applied. Parameters are $C_{CFL}=0.9$  and output time $t_{out} = 0.2$.}\label{eulerdensity}
\end{figure}

\subsection{The Blast wave interaction problem}
Here the one-dimensional Euler equations for an ideal gas with $\gamma = 1.4$ in the domain $[0, 1]$ is solved. The initial condition, in terms of non-conservative variables is given by
\begin{eqnarray}
\begin{array}{c}
\mathbf{W}(x,0) =
\left\{
\begin{array}{cc}
\mathbf{W}_L \;,& x<0.1\;,\\
\mathbf{W}_C \;,& 0.1<x<0.9\;,\\
\mathbf{W}_R \;,& 0.9<x\;,\\
\end{array}
\right.
\end{array}
\end{eqnarray}
with 
\begin{eqnarray}
\begin{array}{c}

\mathbf{W}_L=
\left[
\begin{array}{c}
1\\
0 \\
1000
\end{array}
\right]
\;,
\mathbf{W}_C=
\left[
\begin{array}{c}
1\\
0 \\
0.01
\end{array}
\right]\;,
\mathbf{W}_R=
\left[
\begin{array}{c}
1\\
0 \\
100
\end{array}
\right]\;.
\end{array}
\end{eqnarray}
See \cite{Woodward:1984a} for further details.  The purpose of this test is the assessment of the present strategy for boundary conditions in the case of solid wall boundary conditions.  In this test multiple wave-boundary interaction occurs.    As it has become conventional from \cite{Woodward:1984a}, we solve the problem up to $t_{out}=0.038$ with $800$ cells and reverse boundary conditions. Here we construct non-conservative prescribed functions through interpolations like (\ref{interp:inverseproblems:eq-1}). So two functions of time are available 
\begin{eqnarray}
\begin{array}{c}

\mathbf{\tilde{W}}_L(t) = \mathbf{W}_1^{n-1} + \frac{(t-t^n)}{\Delta t}(\mathbf{W}_1^n-\mathbf{W}_1^{n-1})

\;, \\
\mathbf{\tilde{W}}_R(t) = \mathbf{W}_{N_{int}}^{n-1} + \frac{(t-t^n)}{\Delta t}(\mathbf{W}_{N_{int}}^n-\mathbf{W}_{N_{int}}^{n-1}) \;,

\end{array}
\end{eqnarray}
where $\mathbf{W}_{1}^{n}$ and $\mathbf{W}_{1}^{n-1}$ stand by the non-conservative variables associated to the first cell at both time  steps $t^{n-1}$ and $t^{n}$, respectively.  Similarly, $ \mathbf{W}_{N_{int}}^{n-1} $ and $ \mathbf{W}_{N_{int}}^{n} $ corresponds to the  non-conservative variables associated to the last cell at both times $t^{n-1}$ and $t^{n}$. Additionally, we impose the condition of solid wall $ (\mathbf{\tilde{W}}_L(t))_2 = -(\mathbf{\tilde{W}}_L(t))_2$ and $ (\mathbf{\tilde{W}}_R(t))_2 = -(\mathbf{\tilde{W}}_R(t))_2$. Additionally, we have used the combination $N=\bar{M}=3$ and $L=1.5$ then $\eta = 0.66$, so in virtue of proposition \ref{proposition:space-time:discretization}, it ensures the scheme to be globally stable.  A reference solution is obtained by using the ADER-HEOC solver of third order of accuracy with $3000$ cells, see \cite{Castro:2008a} for further details concerning this solver. Figure \ref{fig:Blast}, shows the result for second and third orders of accuracy, we observe that reverse problems provides very good approximations with respect to the reference solution.

\begin{figure}
\begin{center}
\includegraphics[scale=0.6]{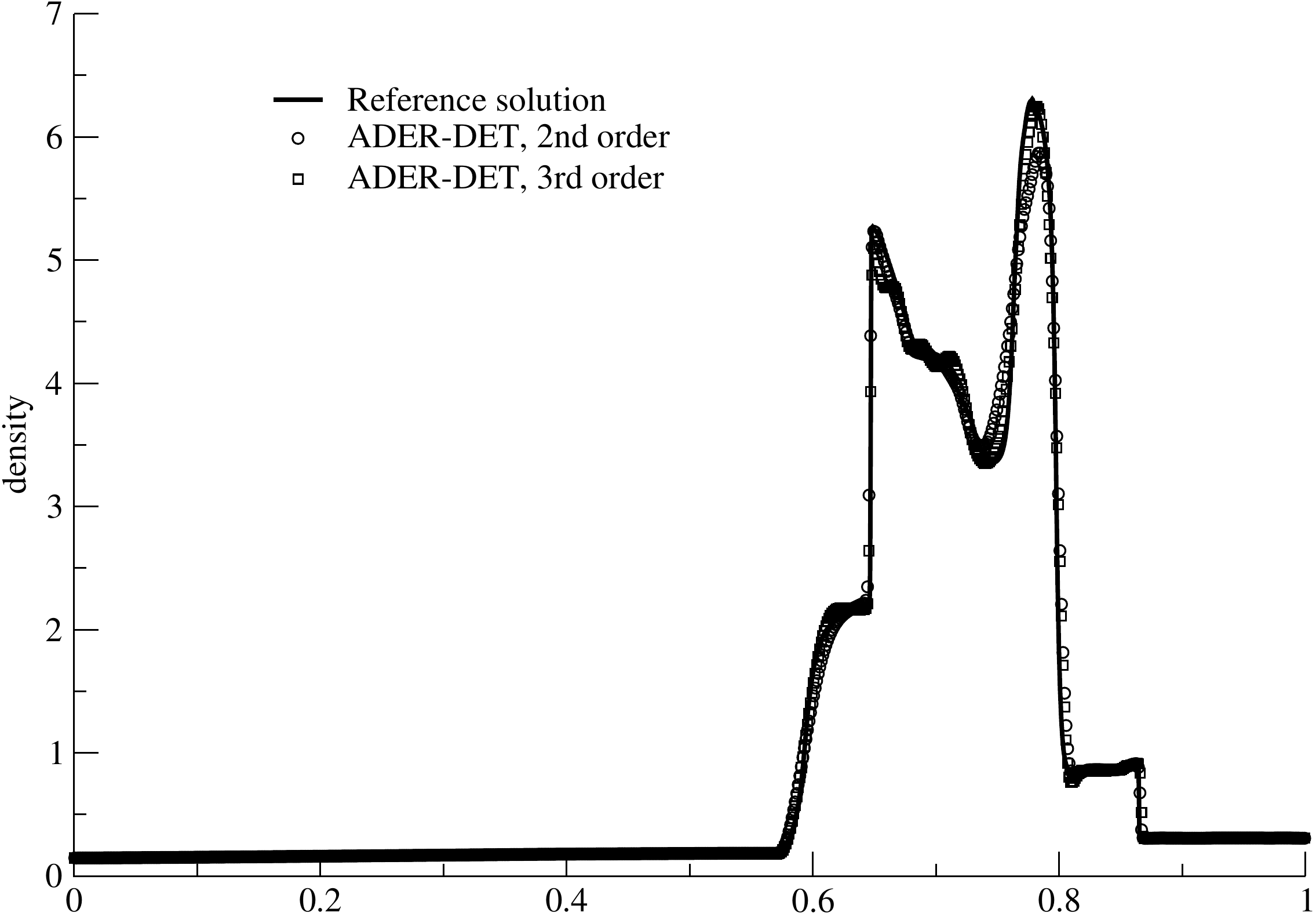}
\end{center}
\caption{Blast wave interaction. Numerical solution of second and third order of accuracy. Output time $t_{out}=0.038$ using $800$ cells.}\label{fig:Blast}
\end{figure}

\section{Conclusions}\label{summary}
In this work a strategy to implement Dirichlet boundary conditions
in the context of ADER schemes for hyperbolic conservation laws has been proposed. 
The strategy  requires the acknowledgment of the information at boundary in terms of a state vector, when information at boundary  is   not prescribed, or partially provided, we proposed a strategy to approximate that information.  The use of this information is twofold; first, to build a Riemann problems at extremes of the computational domain and so, as usual, the wave propagation at boundaries is provided through the solution of these problems.  Second, to compute ghost cells, they are used to get the stencils for cells near to boundaries which is needed in the reconstruction procedure, a key ingredient to build ADER schemes.  Ghost cells are computed from auxiliary problems called reverse problems.

Reverse problems are build from the governing equation of the conservation law, called here  interior problem, and from information at boundaries. So in this sense the ghost cells contains a more physical meaning than other approaches like extrapolations.  Notice that the present approach can be seen as a numerical version of the inverse Lax-Wendroff procedure, but Taylor expansion and the Chauchy-Kowalewski procedure are not required.  In turn, reverse problems are solved numerically by using a conventional second order scheme for hyperbolic balance laws on a suitable mesh, which is built to reproduce the accuracy of the scheme for interior problems.

We have proved analytically that for the scalar case the present strategy is able to reconcile the stability and  accuracy of the method for the interior problem. Furthermore, we have obtained a criterion to select the suitable mesh for the reverse problems in order to obtain a numerical schemes which is stable for both the reverse and interior problems. The numerical method to solve the interior problems was the ADER-DET. However, the present boundary treatment has been designed to be implemented in any high order scheme of the family of ADER methods.
  
 We have dealt the issues of partial acknowledgment of information at boundary as in the case of no prescribed outflow boundaries and the case in which interior problems have no an invertible  flux.  We have solved; linear advection  equation; linear advection system,  a hyperbolic problem with a non-invertible Jacobian matrix and the Euler equations.  Empirical convergence rate assessments and CPU time comparison with the inverse Lax-Wendroff approach, were carried out for some of them. We have obtained accuracy in space and time up to fifth oder of accuracy and the scheme has resulted to be comparable with the well known inverse Lax-Wendroff procedure.
Thus the present boundary condition treatment in the context of high-order ADER numerical schemes is a simple and feasible strategy in terms of efficiency and accuracy. Extension to high-dimension and for complex geometries will be the issues of a future work. 

\section*{Acknowledgements}

The author thanks FONDECYT in the frame of the research project
FONDECYT Postdoctorado 2016, number 3160743.

\section*{References}

\bibliographystyle{plain}
\bibliography{ref} 

\appendix

\section{A simple strategy to carry out the inverse Lax-Wendroff procedure}\label{section:inverse-LW}
Conventional Lax-Wendroff or Cauchy-Kowalewky procedure consists of providing time-derivatives in terms of spatial-derivatives.  This is carried out by a repetitively differentiation of of the governing equations.  In opposite, inverse Lax-Wendroff procedure provides space-derivatives in terms of time-derivatives.  In general the Cauchy-Kowalewsky procedure can be cumbersome for complex conservations laws and symbolic software manipulators may be needed.  Due to the fact that Cauchy-Kowalewsky procedure is commonly available we use it to construct a strategy for the inverse Lax-Wendroff procedure.

To set the problem in the context of applications to boundary conditions and without loss of generality, let us consider an inlet boundary condition. Then ghost cells can be approximated by using a Taylor series expansion in space as follows
\begin{eqnarray}
\begin{array}{c}
\mathbf{Q}(x,t) =\mathbf{Q}(x_L,t) +\displaystyle \sum_{k=1}^{M}\frac{ (x-x_L)^k }{k!} \partial_x^{(k)} \mathbf{Q}(x_L,t) \;.
\end{array}
\end{eqnarray}
Let us assume, as usual, that  time derivatives are obtained in terms of spatial derivatives through the Cauchy-Kowalewsky functionals 
\begin{eqnarray}
\begin{array}{c}
\partial_t^{(k)} \mathbf{Q}(x,t) = \mathbf{G}^{k}(\mathbf{Q}(x,t),...,\partial_x^{(k)}\mathbf{Q}(x,t)) \;.
\end{array}
\end{eqnarray}
Then, given  $ \partial_t^{(l)} \mathbf{Q}(x_L,t) = \mathbf{G}_L^{(l)}(t)$, $l=1,...,M$,  we look for values $\partial_x^{(j)}\mathbf{Q}(x_L,t)$ such that 
\begin{eqnarray}
\begin{array}{c}
\mathbf{G}_L^{(l)}(t) - \mathbf{G}^{k}(\mathbf{Q}(x_L,t),...,\partial_x^{(k)}\mathbf{Q}(x_L,t)) = \mathbf{0} \;,
\end{array}
\end{eqnarray}
for all $k=1,...,M.$ Thus an algebraic equation is then constructed.

\end{document}